\documentclass{amsart}

\usepackage{hyperref}
\hypersetup{
  colorlinks=true, 
   linkcolor=red,  
}

\usepackage{amsmath}
\usepackage{amssymb}
\usepackage{amsfonts}
\usepackage{amsopn}
\usepackage{amsthm}
\usepackage{amscd}
\usepackage[all]{xy}
\usepackage{mathrsfs}
\usepackage{mathtools}
\usepackage{verbatim}

\usepackage[section]{placeins}

\swapnumbers
\theoremstyle{plain}
\newtheorem{thm}{Theorem}[section]
\newtheorem{lem}[thm]{Lemma}
\newtheorem{cor}[thm]{Corollary}
\newtheorem{prop}[thm]{Proposition}

\theoremstyle{definition}
\newtheorem{exs}[thm]{Examples}
\newtheorem{ex}[thm]{Example}
\newtheorem{rem}[thm]{Remark}
\newtheorem{dfn}[thm]{Definition}

\newtheoremstyle{nttstyle}
  {3pt}
  {3pt}
  {}
  {}
  {\bf}
  {.}
  {3pt}
   {}     
  {}

\theoremstyle{nttstyle}

\newtheorem{nttmult}[thm]{Multiplication}
\newtheorem{nttsteen}[thm]{Steenrod operations}
\newtheorem{nttchern}[thm]{Chern classes}

\newtheorem{ntt}[thm]{}

\newcommand{\zz}{\mathbb{Z}} 
 
\newcommand{\nn}{\mathbb{N}}
\newcommand{\qq}{\mathbb{Q}}

\newcommand{\ff}{\mathbb{F}}
\newcommand{\id}{\mathrm{id}}
\newcommand{\ii}{\iota}
\newcommand{\op}{\mathrm{op}}

\newcommand{\A}{\mathrm{A}}

\newcommand{\D}{\mathrm{D}}
\newcommand{\E}{\mathrm{E}}
\newcommand{\F}{\mathrm{F}_4}
\newcommand{\U}{\mathcal{U}}

\DeclareMathOperator{\Sym}{\mathrm{Sym}}
\DeclareMathOperator{\SB}{\mathrm{SB}}

\DeclareMathOperator{\Hom}{\mathrm{Hom}}
\DeclareMathOperator{\End}{\mathrm{End}}

\DeclareMathOperator{\PGL}{\mathrm{PGL}}

\DeclareMathOperator{\codim}{\mathrm{codim}}
\DeclareMathOperator{\car}{char}

\newcommand{\pr}{\mathrm{pr}}
\newcommand{\res}{\mathrm{res}}

\newcommand{\BX}{\overline{X}}
\newcommand{\BY}{\overline{Y}}

\newcommand{\pt}{\mathrm{pt}}
\newcommand{\an}{\mathrm{an}}

\newcommand{\SH}{\mathrm{SH}}

\newcommand{\Gm}{\mathbb{G}_m}
\newcommand{\Gtwo}{\mathrm{G}_2}

\DeclareMathOperator{\Spec}{\mathrm{Spec}}
\DeclareMathOperator{\CH}{\mathrm{CH}}
\DeclareMathOperator{\Ch}{\mathrm{Ch}}

\DeclareMathOperator{\ind}{\mathrm{ind}}

\newcommand{\cG}{\mathcal{G}}

\numberwithin{table}{section}
\numberwithin{equation}{section}

\newcommand{\ksep}{k_{{\mathrm{sep}}}}
\newcommand{\Q}{\qq}
\newcommand{\Z}{\zz}

\DeclareMathOperator{\Spin}{Spin}

\newcommand{\M}{\mathscr{M}}

\newcommand{\Lb}{\bar{L}}

\newcommand{\et}{{\text{\'et}}}
\newcommand{\nr}{{\text{nr}}}

\newcommand{\darkrad}{0.17}
\newcommand{\lrad}{0.4}
\newcommand{\rb}[1]{\raisebox{0.1in}[0pt]{#1}}

\title[Shells of twisted flag varieties]{Shells of twisted flag varieties and
the Rost invariant}
\author{S. Garibaldi
\and V. Petrov 
\and N. Semenov}

\address{Garibaldi: Institute for Pure and Applied Mathematics, UCLA, 460 Portola Plaza, Box 957121, Los Angeles, California 90095-7121, USA}
\email{skip@garibaldibros.com}

\address{Petrov: St.~Petersburg Department of Steklov Mathematical Institute,
Russian Academy of Sciences,
Fontanka 27,
191023 St.~Petersburg,
Russia \newline and \newline Chebyshev Laboratory,
St. Petersburg State University,
14th Line, 29b, 199178 St.~Petersburg,
Russia
}
\email{victorapetrov@googlemail.com}

\address{Semenov: 
Institut f\"ur Mathematik, Johannes Gutenberg-Universit\"at
Mainz, Staudingerweg 9, D-55128 Mainz, Germany}
\email{semenov@uni-mainz.de}

\begin{document}


\begin{abstract}
We introduce two new general methods to compute
the Chow motives of twisted flag varieties
and settle a 20-year-old conjecture of Markus Rost about
the Rost invariant for groups of type $\E_7$.
\end{abstract}

\maketitle

\setcounter{tocdepth}{1}

\tableofcontents

\section{Introduction}                                         
A twisted flag variety over an arbitrary field --- such as a quadric or a Severi-Brauer variety --- ``typically" has no rational points.  In such a case, the Chow motive of the variety provides one of the few tools for studying its geometry.               
In this article, we introduce two new general methods to compute the Chow motive of such a variety.

Chow motives were introduced by Alexander Grothendieck, and they have since become
a fundamental tool for investigating the structure of algebraic
varieties.  Computing 
 Chow motives has also proved to be valuable for addressing questions on other topics.  As examples, we mention:
\begin{itemize}
\item Voevodsky's proof of the Milnor conjecture relied on Rost's computation of the motive of a Pfister quadric.
\item Progress on the long-standing Kaplansky problem about possible values of the $u$-invariant relied on the Chow motives and Chow rings of quadrics and quadratic Grassmannians, see \cite{Vi07}.
\item The structure of the powers of the fundamental ideal in Witt rings (\cite{Ka04}),
excellent connections in the motives of quadrics (\cite{Vi10}), the proof of the hyperbolicity conjecture for orthogonal involutions (\cite{Ka09b}),
the proof of Hoffmann's conjecture on the higher Witt indexes of quadratic forms (\cite{Ka03}),
an incompressibility results as in \cite{Ka09} rely on analyzing the Chow motive of generalized Severi-Brauer varieties
and products of quadrics.
\item The motive of the Borel variety for $\E_8$ as in \cite{PSZ} 
and \cite{PS12} was used to answer a question of Serre \cite{Sem09} and other questions about finite subgroups of $\E_8$ \cite{GS09}.
\item Karpenko and others proved many results on isotropy of involutions on central simple algebras by analyzing motivic decompositions of projective homogeneous varieties, culminating in the paper \cite{KaZ} by Karpenko and Zhykhovich which also treats unitary Grassmannians.
\end{itemize}
We illustrate the power of the general methods developed in this paper by calculating the Chow motives of some twisted flag varieties for which this computation was previously not possible.

\subsection*{First method: shells} Our first method (\S\ref{shells.sec}) generalizes Vishik's shells of quadratic forms from \cite{Vi03}
and extends Karpenko's result on the upper motives from \cite{Ka09}.
Karpenko proved that any indecomposable direct summand of the Chow motive of a twisted flag variety
of inner type is isomorphic to a Tate twist of the upper motive of another
twisted flag variety.
Thus, the study of motivic decompositions of twisted flag varieties is reduced
to the study of the upper motives.  For generically split varieties, the structure of upper motives was determined in \cite{PSZ}, but apart from this case there are unfortunately very few results about the structure of upper motives.

Our method is aimed to address this lack.
It turns out that one can subdivide algebraic cycles on twisted flag
varieties into several classes, called {\it shells}.  Direct summands of the Chow motives of twisted flag varieties
starting in the same shell are of the same nature,
and our first main result (Theorem \ref{t1}) asserts that one can shift these direct summands inside shells under the condition of existence of a certain
cycle $\alpha$.  This condition is essential;  the theorem does not hold without it.  Moreover, this condition is automatically satisfied for quadrics, so the notion of shell given here generalizes Vishik's shells for quadrics.

We use the shell method to prove that certain ``big'' direct summands are indecomposable.

\subsection*{Second method}
Our second method (Theorem~\ref{tchme}) is based on a formula of Chernousov and Merkurjev
from \cite{CMe06}. This method provides a broad generalization of the generic point diagram,
which is a standard tool in the theory of Chow motives. Namely, if $\alpha$
is a cycle on a twisted flag $G$-variety $\BX$, which
is defined over the function field of a twisted flag $G$-variety $X'$ (where $G$
is a semisimple algebraic group), then the cycle
$\alpha\times 1 +\beta$ on the product $\BX\times\BX'$ is defined over the base field for a certain cycle $\beta$.
Unfortunately, the generic point diagram does not give a precise formula for $\beta$, and one works with
undefined coefficients. Our method allows one to compute $\beta$ explicitly --- see Example~\ref{ex65}
for an illustration.  Moreover,
our result provides an explicit algorithmic description of all rational cycles on the product $\BX\times\BX'$, and in particular of all rational projectors (by taking $X = X'$).
We use this method in Section~\ref{s8} to construct new projectors for Chow motives of $\E_6$-varieties. We do not know any
alternative way of doing this.

Our two methods are complementary to each other. The first method is designed to eliminate certain
motivic decomposition types, and the second one to prove that the remaining decomposition types
are realizable.

\subsection*{Algorithm}
We also provide in section \ref{secSteenrod} an algorithm to calculate the multiplication table for the
equivariant and ordinary Chow rings of twisted flag varieties and the equivariant and
ordinary Steenrod operations modulo $p$.
This algorithm is a generalization of one described by Knutson and Tao in \cite{KT03} for Grassmannians.
This section of the paper can be read independently of the rest.

\subsection*{Applications}
To demonstrate our methods, we provide a complete classification of
motivic decompositions of all
twisted flag varieties of inner type $\E_6$ (see Section~\ref{s8}).

We also use our methods to prove results relating isotropy of certain groups with values of the Rost invariant.  Recall that the Rost invariant is a map 
\begin{equation} \label{rost.inv}
r_G \!: H^1(k, G) \to H^3(k, \Q/\Z(2))
\end{equation}
that is functorial in $k$ and defined for every simple simply connected algebraic group $G$.  For such $G$, it is, roughly speaking, the first nonzero invariant \cite[\S31]{KMRT}, and it generates the group of invariants with codomain $H^3(*, \Q/\Z(2))$ \cite{GMS}.  It plays an important role in the study of quadratic forms (where it is known as the Arason invariant) and it was crucial in Bayer and Parimala's proof of the Hasse Principle Conjecture II for classical groups in \cite{BP:H} and the proof of all known cases of the conjecture for exceptional groups.  Generally speaking, such cohomological invariants are important because they provide a way to transform questions about the pointed set $H^1(k, G)$ into questions about an abelian group like $H^3(k, \Q/\Z(2))$.  For the definition and basic properties of the Rost invariant, see the portion of the book \cite{GMS} written by Merkurjev and Garibaldi.

More than 20 years ago, Markus Rost conjectured\footnote{Letter to Jean-Pierre Serre, dated November 1992.}
that the Rost invariant for
groups of type $\E_7$ detects rationality of  parabolic subgroups.  We combine the motivic techniques developed here to prove his conjecture, which appears as Theorem \ref{prope6} below.  Moreover, Tonny Springer raised the general question of 
possible relations between the Rost invariant and rationality of parabolic
subgroups for groups of type $\E_7$ in \cite{Sp06}; 
Theorems \ref{prope6} and  \ref{pro816} settle this.

We prove analogous results for all parabolic subgroups
of groups of type $\E_l$, including $\E_8$ at the prime 3 (Theorems \ref{prope6}, \ref{pro816}, and \ref{E8symb}). Note that usually results for primes bigger than 2 are substantially more complicated than the version for 2. Our methods work in the same manner for all primes.

Further, our results on the Rost invariant give some classification results for algebraic groups over function fields of $p$-adic curves (Corollaries~\ref{co815} and \ref{E8.QpC}).

We remark that we do not use the second method to prove these results about the
Rost invariant, nor do we need the full generality of shells.  (The full generality of both methods is essential,
however, for the decomposition of the $\E_6$-varieties in Section \ref{s8}.)  For our Rost invariant results, it suffices to use 
just the first shell, the algorithm to compute the Steenrod operations (applied only to $\E_7$-varieties),
Karpenko's upper motive, and Proposition~\ref{lekarp2} below due to De Clercq (which is a generalization
of Karpenko's \cite[Prop.~4.6]{Ka09b}). On the other hand, we believe that we would
not have found any of the proofs in the present article without developing the general methods first.

In summary, the applications to groups of type $\E_l$ are substantially
stronger than previous results, which were obtained using algebraic and cohomological 
techniques.  Furthermore, our proofs are more ``conceptual" in that the difficult work is part of the new motivic
techniques, which are completely general, as opposed to being specific to the group.
Also, all results of the present article apart from a few technical statements hold over fields of any characteristic,
even if the characteristic equals a torsion prime of the group.

Finally, our methods are also new for groups
of classical types, and can be applied to study, for example, the generalized Severi-Brauer varieties and quadratic or symplectic
Grassmannians.

\section{Background on algebraic groups and motives} \label{background}

\subsection*{Algebraic groups}
Detailed information on algebraic groups can be found in \cite{Spri} and \cite{KMRT}.

\begin{ntt}
Let $k$ be a field and $G$ a semisimple linear algebraic group of \emph{inner} type
over $k$. We write $\Phi$ for the root system of $G$, $\Phi^+$ resp.~$\Phi^-$
for the set of positive resp.~negative roots, and $\Delta$ for the Dynkin diagram of $G$ and by abuse of notation also for the set of vertices or simple roots.  We enumerate the simple roots following Bourbaki, and we recall the precise numbering for groups of
type $\E$ in \eqref{E6.roots} and \eqref{E78.roots} below.

For every subset $\Theta$ of $\Delta$, there is a projective homogeneous $G$-variety $X_\Theta$ of parabolic subgroups of type $\Theta$; these are the \emph{twisted flag varieties} of $G$.
We normalize the notation so that $X_\emptyset = \Spec k$, $X_{\{ \alpha \}}$
corresponds to a maximal parabolic subgroup, and $X_\Delta$ is the Borel variety.
We occasionally omit the braces and write $X_{1,2}$ for $X_{\{1,2\}}$, for example.
If $G$ is a split group, then in the same way we write $P_\Theta$
for a standard parabolic subgroup of type $\Theta$ so that $X_\Theta \simeq G/P_\Theta$.
\end{ntt}

\begin{ntt}
The {\it Tits index} of the group $G$ is the set of vertices $i\in\Delta$ such that the variety $X_i$ has a
rational point; we draw it by circling those vertices in the Dynkin diagram of $G$. The possible Tits indexes have been determined in \cite{Ti66}, or see \cite[\S17]{Spri} or \cite{PSt}.

Let $S$ be a maximal $k$-split torus of $G$ and following \cite{Ti66} put $G_\an$ for the derived subgroup $[Z_G(S),Z_G(S)]$ of the reductive group $Z_G(S)$.  The subgroup $G_\an$ is called the \emph{semisimple anisotropic kernel} of $G$.  It is semisimple and $k$-anisotropic, and it is uniquely determined by $G$ up to $G(k)$-conjugacy because the maximal $k$-split tori are conjugate under $G(k)$.
The Dynkin type of $G_\an$ equals $\Delta\setminus L$, where $L$ is the Tits index of $G$, and
the Tits index of $G_\an$ is empty.
\end{ntt}

\subsection*{Rost invariant} The Rost invariant from \eqref{rost.inv} takes values in the group 
 $H^3(k, \qq/\zz(2))$, which is defined to be the direct sum over all primes $p$ of $\varinjlim_m H^3(k, \zz/p^m\zz(2))$, where 
\[
H^{d+1}(k, \zz/p^m\zz(d)) := \begin{cases}
H^{d+1}(k, \mu_{p^m}^{\otimes d})& \text{if $\car k \ne p$;} \\
H^1(k, K_d(\ksep)/p^m)&\text{if $\car k = p$,}
\end{cases}
\]
the groups on the right are Galois cohomology, and $\ksep$ is a separable closure of $k$, see \cite[pp.~151--154]{GMS}.
One defines $H^3(k, \zz/n\zz(2))$ analogously for composite $n$, and it is naturally identified with the $n$-torsion in $H^3(k, \Q/\Z(2))$.  Note that $H^1(k, \Z/n\Z(0))$ is the Galois cohomology group $H^1(k, \Z/n\Z)$ regardless of the characteristic of $k$.

There is a cup product map
\[
\left( \times^d K_1(k) \right) \times H^1(k, \Z/n\Z) \to H^{d+1}(k, \Z/n\Z(d)),
\]
and we call elements of the image (including zero) \emph{symbols}.

\subsection*{Chow motives}
\begin{ntt}
By a variety we always mean a reduced separated scheme of finite type over a field.
Let $p$ be a prime number.
For a smooth projective variety $X$ over $k$, we write $\CH(X)$ for its Chow ring modulo rational equivalence
and set $\Ch(X) :=\CH(X)\otimes\ff_p$. We write $\deg$ for the degree map
$\CH_0(X)\to\zz$, and for a field extension $K/k$ we write $X_K$ for the corresponding
extension of scalars. A cycle $\alpha\in\Ch(X_K)$ is called {\it rational} (with respect to $k$),
if it lies in the image of the restriction map $\Ch(X)\to\Ch(X_K)$. A subgroup of $\Ch(X_K)$ is called
rational if all its elements are rational.
\end{ntt}

\begin{ntt}                                        
We consider the category of Chow motives over $k$ with $\ff_p$-coefficients
(see \cite{Ma68} or \cite[\S64]{EKM}). The motive of a variety $X$ is denoted by $\M(X)$. For a field extension $K/k$ and
a motive $M$ we denote by $M_K$
the respective extension of scalars. The shifts of Tate motives
are denoted by $\ff_p(i)$.
\end{ntt}

\begin{ntt}
Let $X$ be a smooth projective irreducible variety over $k$.
A motive $M=(X,\pi)$ for a projector $\pi$ is called \emph{geometrically} (resp., \emph{generically}) \emph{split},
if over some field extension $F$ of $k$ (resp., over $k(X)$) it is isomorphic to a finite sum $\bigoplus_{i\in I}\ff_p(i)$
of Tate motives for some multiset of non-negative indexes $I$. The field $F$ is called a \emph{splitting field} of $M$, and
for a cycle $\alpha\in\Ch(X)$ we set $\bar\alpha=\alpha_F$.

For a twisted flag variety $X$,
the motive $\M(X)$ is geometrically split (see \cite[Theorem~2.1]{Ko91}), and
we denote by $\BX$ the scalar extension of $X$ to a splitting field of its motive.
The Chow ring of $\BX$ is independent of the choice of splitting field. Its structure is explicitly
described in Section~\ref{secSteenrod}.

We define the {\it Poincar\'e polynomial} of a geometrically split motive $M$ by the formula $P(M,t)=\sum_{i\in I}t^i\in\zz[t]$.
The Poincar\'e polynomial is independent of the choice of a splitting field $F$.
We define the \emph{dimension} of $M$ to be $\dim M := \max I-\min I$.
An explicit formula for $P(\M(X),t)$ for a twisted flag variety $X$ is given in \cite[\S2]{PS10}.

In a similar way we define the Poincar\'e polynomial of a finite-dimensional
$\zz_{\ge 0}$-graded vector space $A^*$ as $P(A^*,t)=\sum_{i\ge 0}\dim A^i\cdot t^i$.
\end{ntt}

\begin{ntt}[Krull-Schmidt]
For a motive $M$ we say that the Krull-Schmidt theorem holds for $M$, if
for any two motivic decompositions of $M$
$$M\simeq M_1\oplus\ldots\oplus M_a\simeq N_1\oplus\ldots\oplus N_b$$
with all motives $M_i$, $N_j$ indecomposable,
we have $a=b$ and there exists a permutation $\sigma$ of $\{1,\ldots,a\}$ such that
$M_i\simeq N_{\sigma(i)}$ for all $i=1,\ldots,a$.

By \cite{CMe06} the Krull-Schmidt theorem holds for all twisted flag varieties in the
category of motives with $\mathbb{F}_p$-coefficients.
\end{ntt}


\section{Karpenko's theorem on upper motives and generic points of motives}\label{uppp}
Let $G$ be a semisimple linear algebraic group of inner type over a field $k$,
$X$ be a twisted $G$-homogeneous flag variety over $k$, $p$ a prime number,
and $\U(X)$ the (unique) indecomposable direct summand of the Chow motive of $X$
with $\ff_p$-coefficients such that $\Ch^0(\U(X)) \ne 0$.
The set of isomorphism classes of the motives $\U(Y)$ for all
twisted $G$-homogeneous flag varieties $Y$ is called the set
of {\it upper motives} of $G$.

\begin{prop}[Karpenko, {\cite[Theorem~3.5]{Ka09}}]\label{lekarp1}
Every indecomposable summand of $X$ is isomorphic to a Tate shift of an upper
motive $\U(Y)$, where $Y$ runs over all twisted $G$-homogeneous flag varieties over $k$ such that the Tits index of $G_{k(Y)}$
contains the Tits index of $G_{k(X)}$. $\hfill\qed$
\end{prop}

We also need a result of De Clercq:

\begin{prop}[De Clercq, {\cite[Thm.~1.1]{DC10}}]\label{lekarp2}
Let $X$ and $Y$ be twisted flag varieties, and let $M$ and $N$ be direct summands of $\M(X)$ and $\M(Y)$ respectively.   
If $N_{k(X)}$ is an indecomposable direct summand of $M_{k(X)}$ and
every cycle in $\Ch(\overline{Y\times X})$ which is defined over $k(X)(Y_{k(X)})$ is defined over $k(Y)$,
then $N$ is a direct summand of $M$.$\hfill\qed$
\end{prop}                                            

\begin{dfn} \label{genpt}
Let now $X$ be a smooth projective irreducible variety and $M=(X,\pi)$ a geometrically split motive.
Assume that over a splitting field $F$ of $M$ the motive $M_F\simeq\bigoplus_{i\in I\cup\{l\}}\ff_p(i)$
for a multiset $I$ of indexes such that every $i\in I$ is bigger than $l$. Then $\Ch^l(M_F) \simeq \ff_p$, and 
any nonzero element in $\Ch^l(M_F)$ is called a
\emph{generic point} of $M$; we abuse language and write ``the" generic point.
\end{dfn}

The following two results are well-known:

\begin{lem}
Let $X$ be a twisted flag $G$-variety.
The generic point of a direct summand of the motive of $X$
is rational (i.e., defined over $k$).
\end{lem}
\begin{proof}
Follows from Prop.~\ref{lekarp1}.
\end{proof}

\begin{lem}\label{genpt2}
If, in the notation of Definition~\ref{genpt}, $M$ is generically split, then the generic point of $M$
is rational.
\end{lem}
\begin{proof}
Let in the notation of Definition~\ref{genpt} $M=(X,\pi)$.
The motive $M_{k(X)}$ is a direct sum of shifted Tate motives.
Let $l$ be the smallest integer such that the Tate motive $\ff_p(l)$
is a direct summand of $M_{k(X)}$.
This Tate motive is defined by two cycles $a\in\Ch_l(X_{k(X)})$ and $b\in\Ch^l(M_{k(X)})$
with $\deg(ab)=1$ and in the Sweedler notation
$$\bar\pi=a\times b+x_{(1)}\times x_{(2)}$$ with $\codim x_{(2)}>l$.
We want to show that $b$ is defined over $k$.                                               

Consider $M_{k(X)}$.
By the generic point diagram (\cite[Lemma~1.8]{PSZ})
the cycle $$b\times 1+y_{(1)}\times y_{(2)}$$
with $\codim y_{(1)}<l$ is rational. Hence, the product of the above cycles equals
$$\pt\times b+z_{(1)}\times z_{(2)},$$ where $\codim z_{(1)}<\dim X$ and $\pt$
is the class of a rational point on $X$, and is rational. Taking the push-forward with respect
to the second projection $$X\times X\to X,$$ one sees that $b$ is rational.
\end{proof}

\section{Shells} \label{shells.sec}

The content of this section is a generalization
of the notion of shells for quadratic forms invented by Vishik (see \cite{Vi03})
and Karpenko's result quoted as Proposition \ref{lekarp1} above.  Let $X$ denote a twisted $G$-homogeneous flag variety that is not a point.

\begin{dfn}[big shells]\label{shelldef}
For each set $\Psi$ of vertices of the Dynkin diagram $\Delta$ of $G$,
we put $K_\Psi$ for the function field of the variety $X_\Psi$.
Define the (big) shell $\SH_{\leqslant \Psi}$ of $X$ to be the union for all $i$ of the $b\in\Ch^i(\BX)$
such that
\begin{enumerate}
\item $b$ is defined over $K_\Psi$ and
\item there is an $a\in\Ch_i(\BX)$
defined over $K_\Psi$ such that $\deg(ab) = 1 \in \ff_p$. 
\end{enumerate}
\end{dfn}

Note that each shell is closed under multiplication by elements of $\ff_p^\times$, and that the shells depend on the prime $p$
(even though the Poincar\'e polynomial of $\Ch(\BX)$ does not).

\begin{exs}
A shell $\SH_{\leqslant \Psi}$ is nonempty iff $X \times K_\Psi$ has a zero-cycle of degree not divisible by $p$.  Consequently, for $\Theta \subseteq \Delta$ such that $X = X_\Theta$, the shell $\SH_{\le \Theta}$ is always nonempty.  We call it the \emph{first shell} of $X$.

For $\Psi = \Delta$ or more generally for any $\Psi$ such that $K_\Psi$ splits $G$, $\SH_{\leqslant \Psi}$ is the set of nonzero homogeneous elements of $\Ch(\BX)$, because the pairing $(a, b) \mapsto \deg(ab)$ on $\Ch(\BX)$ is non-degenerate.  We call this the \emph{last shell} of $X$.
\end{exs}

\begin{ex}
If $\Psi \subseteq \Psi'$, then there is a natural surjection $X_{\Psi'} \to X_{\Psi}$, hence an inclusion $K_\Psi \subseteq K_{\Psi'}$, so $\SH_{\leqslant \Psi} \subseteq \SH_{\leqslant \Psi'}$.  
\end{ex}

\begin{dfn}
A {\it $(p,\Theta)$-index} $S$ of $G$ is a set $\Theta \subseteq S \subseteq \Delta$ such that there exists an extension $L$ of $k$ with the properties that (a) the Tits index of $G_L$ is $S$ and (b) every finite extension of $L$ has degree a power of $p$.

The set of $(p,\Theta)$-indexes of $G$ is a partially ordered set (by inclusion). We define the {\it height} of $X=X_\Theta$
as the maximal number of elements in a chain of $(p,\Theta)$-indexes of $G$.  
\end{dfn}

\begin{prop}\label{shell.include}
Let $\{ S_i \mid i \in I \}$ be the set of $(p, \Theta)$-indexes of $G$.
\begin{enumerate}
\item \label{shell.include.1}
Every nonempty shell on $X=X_\Theta$ equals $\SH_{\leqslant \cap_{j \in J} S_j}$  for some subset $J \subseteq I$.
\item \label{shell.include.2}
Suppose that $\Theta \subseteq \Psi, \Psi' \subseteq \Delta$ are such that, for every $i$, if $\Psi' \subset S_i$ then $\Psi \subset S_i$.  Then
$\SH_{\leqslant \Psi} \subseteq \SH_{\leqslant \Psi'}$.
\end{enumerate}
\end{prop}

\begin{proof}
First suppose that $\Psi, \Psi' \subseteq \Delta$ are such that there exists a finite extension $F$ of $K_{\Psi'}$ of degree not divisible by $p$ such that $X_\Psi(F) \neq \emptyset$.  Put $E$ for the fraction field of the integral domain $F \otimes K_\Psi$; it is the function field of the $F$-variety $X_\Psi \times F$.  The diagram
\[
\xymatrix@R=1pc{
\Ch(X \times K_\Psi) \ar[dr] \ar[r] & \Ch(X \times E) \ar[d] & \Ch(X \times F) \ar[l]_{\res_{E/F}} & \Ch(X \times K_{\Psi'}) \ar[l]_{\res_{F/K_{\Psi'}}} \ar[lld] \\
& \Ch(\BX)
}
\]
commutes where all the arrows are scalar extension.  Each $b \in \SH_{\leqslant \Psi}$ is the image of some $b_0 \in \Ch(X \times K_\Psi)$ and we consider the image of $b_0$ in $\Ch(X \times E)$.  As $X_\Psi(F)$ is nonempty, $E$ is a purely transcendental extension of $F$, hence $\res_{E/F}$ is an isomorphism by \cite[Prop.~2.1.8]{FGV}.  The arrow $\res_{F/K_{\Psi'}}$ is also an isomorphism, so the image of $b_0$ in $\Ch(X \times K_{\Psi'})$ maps to $b \in \SH_{\leqslant \Psi'}$, proving that $\SH_{\leqslant \Psi} \subseteq \SH_{\leqslant \Psi'}$.

For \eqref{shell.include.2}, there is a finite extension $F$ of $K_{\Psi'}$ of degree not divisible by $p$ such that the Tits index of $G_F$ is $S_i$ for some $i$; necessarily this $S_i$ contains $\Psi'$, hence also $\Psi$, i.e., $X_\Psi(F)$ is nonempty.  The previous paragraph gives \eqref{shell.include.2}.

Now let $\Psi$ be an arbitrary subset of $\Delta$ such that $\SH_{\leqslant \Psi}$ is nonempty. Then there is a finite extension $F$ of $K_\Psi$ of degree not divisible by $p$ such that $X_\Theta(F)$ is nonempty and the Tits index of $G_F$ contains both $\Psi$ and $\Theta$, i.e., $X_{\Theta \cup \Psi} \times K_\Psi$ has a zero-cycle of degree not divisible by $p$.  By the previous paragraph, we find that $\SH_{\leqslant \Psi} = \SH_{\leqslant \Psi \cup \Theta}$.

For \eqref{shell.include.1}, let $\Psi_0 \subseteq \Delta$ be such that $\SH_{\leqslant \Psi_0}$ is nonempty.  By the previous paragraph, we may assume that $\Psi_0 \supseteq \Theta$.  Put $\Psi$ for the intersection of all the $S_i$'s containing $\Psi_0$.  By \eqref{shell.include.2}, $\SH_{\leqslant \Psi} = \SH_{\leqslant \Psi_0}$, hence \eqref{shell.include.1}.
\end{proof}

\begin{ex}[Quadrics] \label{quadric.eg}
In \cite{Vi03} Vishik describes a subdivision of the Chow group of projective quadrics into shells.

Let $p=2$ and let $q$ be an anisotropic regular quadratic form over $k$ of dimension $n+2$ and $X$ the
projective quadric given by the equation $q=0$.
Since in the present article we consider the groups of inner type only, we assume that the discriminant
of $q$ is trivial, if $n$ is even.

Let $h\in\Ch^1(\BX)$ be the class of a hyperplane
section of $\BX$ and $l_s$, $s=0,\ldots,[n/2]$, the classes of $s$-dimensional subspaces on $\BX$.
Then the Chow group $\Ch^*(\BX)$ has a basis $$\{h^s,l_s\mid s=0,\ldots,[n/2]\}.$$

Let $i_1 < \cdots < i_r$ be the splitting pattern of $q$ (in the usual sense of \cite[p.~104]{EKM}, as opposed to the variation used in \cite[p.~31]{Vi03}) and set $i_0=0$.
Then the cycles $\{h^s,l_s\mid i_{t-1}\le s\le i_t-1\}$ belong to the shell $t\in\nn$ in the notation of Vishik.

In our notation the cycles $\{h^s,l_s\mid 0\le s< i_t\}$ belong to $\SH_{\leqslant\{i_t\}}$,
and the cycles $\{h^s,l_s\mid i_{t-1}\le s < i_t\}$ belong to $\SH_{\leqslant\{i_t\}}\setminus\SH_{\leqslant\{i_{t-1}\}}$.
\end{ex}

\begin{ex} \label{tower}
Suppose the set of $(p, \Theta)$-indexes of $G$ is contained in $\{ S_1, \ldots, S_r \}$ such that $S_1 \subset S_2 \subset \cdots \subset S_r = \Delta$.  Then by Proposition \ref{shell.include}\eqref{shell.include.1}, the nonempty shells of $S$ are
\[
\SH_{\leqslant \Theta} = \SH_{\leqslant S_1} \subseteq \SH_{\leqslant S_2} \subseteq \cdots \subseteq \SH_{\leqslant S_r} = \SH_{\leqslant \Delta}.
\]
This situation occurs in Example \ref{quadric.eg}, where we find $r$ distinct $(2, \{ 1 \})$-indexes and $r$ distinct shells.
\end{ex}

\begin{dfn}[small shells]
For each set $\Psi$ of vertices of the Dynkin diagram of $G$, we set the \emph{(small) shell} $\SH_\Psi$ to
be the union for all $i$ of the cycles $b\in\Ch^i(\overline X)$ such that $b$ is the generic point of
an indecomposable direct summand $M$ of the motive of $X$ such that $M$ is isomorphic to a Tate twist of $\mathcal{U}(X_\Psi)$.
In particular, by the proof of \cite[Th.~3.5]{Ka09}, $\SH_\Psi\subseteq\SH_{\leqslant\Psi}$.
\end{dfn}

We remark that the big shells reflect the splitting behavior of the group $G$ and the small shells
reflect the motivic behavior of $G$.

We say that a motive $M$ {\it starts} in the shell $\SH_\Psi$ (resp.~in codimension $l$),
if its generic point belongs to $\SH_\Psi$ (resp.~to $\Ch^l(\BX)$).

\begin{lem}\label{lx}
Let $X$ be a smooth projective variety over $k$ and $M$ an indecomposable geometrically split motive with
a splitting field $F$ satisfying the following conditions:
\begin{enumerate}
\item the kernel of the natural map $\End(M)\to\End(M_F)$ consists of nilpotent correspondences;
\item $M_F\simeq\bigoplus_{i\in I\cup\{l\}}\ff_p(i)$ for a multiset $I$ of indexes such that every $i\in I$
is bigger than $l$;
\item there exist two morphisms $\alpha\colon M\to \M(X)$ and $\beta\colon \M(X)\to M$ such that
the composition $\beta\circ\alpha\colon M\to M$ maps over $F$ the generic point of $M_F$ identically
on itself.
\end{enumerate}
Then $M$ is isomorphic to a direct summand of $\M(X)$.
\end{lem}
\begin{proof}
Let $M=(Y,\pi)$ for some smooth projective variety $Y$ over $k$.
Since $M$ is geometrically split, the ring $\End(M_F)$ is finite, and therefore some power, say $n$, of
$\bar\beta\circ\bar\alpha\in\End(M_F)$ is a projector. This projector is non-zero by condition 3).
Since $M$ is indecomposable, this projector must be equal to $\bar\pi$.

Denote $\alpha':=\alpha\circ(\beta\circ\alpha)^{\circ(n-1)}\colon M\to \M(X)$. Then $\bar\beta\circ\bar\alpha'=\bar\pi$.
By condition 1) $\beta\circ\alpha'=\pi+\nu$ for some nilpotent correspondence $\nu$. Denote
$\beta':=(\pi+\nu)^{-1}\circ\beta$. Then $\beta'\circ\alpha'=\pi$, and the lemma follows.
\end{proof}                                                       

The following theorem is known for smooth projective quadrics by \cite{Vi03}.
\begin{thm}\label{t1}
Let $b\in\Ch^l(\BX) \cap \SH_\Psi$ be the generic point of an indecomposable direct summand $M$ of $\M(X)$
and $\alpha\in\Ch^t(\BX)$ a cycle defined over $k$.
If the cycle $b':=b\cdot \alpha$ is in $\SH_{\leqslant\Psi}$,
then there is an indecomposable direct summand $M'$ of $\M(X)$ with generic point $b'$ and isomorphic
to $M(t)$.
\end{thm}

The results of Section~\ref{s8}
below show that one cannot in general weaken any condition of the theorem.

\begin{proof}[Proof of Theorem \ref{t1}]
Set $Y=X_\Psi$.
By assumptions $M$ is isomorphic to $N(l)$ for the upper motive $N=\mathcal{U}(Y)$.
Let $d'$ be a cycle dual to $b'$ in the definition of shells. Then $d'$ is
defined over $k(Y)$.

Define a sequence of morphisms
$$M(t)=N(t+l)\xrightarrow{\alpha}Y(t+l)\xrightarrow{\beta}X\xrightarrow{\gamma}X(t)\xrightarrow{\delta}M(t),$$
where $\alpha$ is an embedding of $N(t+l)$ as a direct summand of $Y(t+l)$,
$$\bar\beta=1\times d'+y_{(1)}\times x_{(2)}\in\Ch_{\dim Y+l+t}(\overline Y\times\BX)$$ with $\codim y_{(1)}>0$,
$$\bar\gamma=(\alpha\times 1)\cdot\Delta_X\in\Ch_{\dim X-t}(\BX\times\BX),$$
and $\delta$ is the projection onto the direct summand.

To finish the proof it suffices to notice that the composition
$\delta\circ\gamma\circ\beta\circ\alpha$ maps $\ff_p(t+l)$ to $\ff_p(t+l)$ identically over $\bar k$
and apply Lemma~\ref{lx}.
\end{proof}

\begin{cor}
Let $b\in\Ch^l(\BX)$ be a rational cycle from the first shell of $X$. Then there is
an indecomposable direct summand of $X$ with generic point $b$ isomorphic
to the $l$-th Tate shift of the upper motive of $X$.$\hfill\qed$
\end{cor}

\section{Multiplication and Steenrod operations}\label{secSteenrod}
\begin{ntt}
Let $G_0$ be a split adjoint semisimple group.  We fix a parabolic subgroup $P$ containing a Borel subgroup $B$ containing a maximal split torus $T$ of $G_0$.  Occasionally we need to perform explicit calculation
in $\CH^*(G_0/P)$ or in $\Ch^*(G_0/P)$ considered as a ring or (in the latter case)
as a module over the Steenrod algebra.  In this section, we provide algorithms for doing so based on passing to the $T$-equivariant cohomology described in \cite{Bri97}, as was done for Grassmannians in \cite{KT03}.
There is a diverse literature studying the equivariant Chow rings, of which we indicate as examples \cite{FK96},
\cite{GKM}, \cite{HHH}, \cite{Ty}, \cite{GHZ}, and the references therein.

It is well-known (see e.g. \cite{Ko91}) that $\CH^*(G_0/P)$ has an additive basis consisting of
the classes of \emph{Schubert subvarieties} $X_w=[\overline{BwP/P}]$, where $w\in W/W_P$, $W$ stands for the Weyl group
of $G_0$ and $W_P$ stands for the Weyl group of (a Levi subgroup of) $P$.
We identify the cosets in $W/W_P$ with their minimal representatives.
The dimension of $X_w$ is $l(w)$, the minimal length of $w$ in the simple reflections.

Sometimes it is more convenient to enumerate the generators as
$$Z_w=X_{w_0ww_0^P},$$ where $w_0$ is the longest element of $W$ and $w_0^P$
is the longest element of $W_P$. Then the codimension of $Z_w$ is $l(w)$, in
particular, we have
$$
\pt=X_1=Z_{w_0w_0^P}.
$$
Note that
$Z_w$ is the \emph{Poincar\'e dual} to $X_w$ in the sense that
$$
X_u\cdot Z_w=\delta_{u,w} \,\pt.
$$
If $Q\subseteq P$ is another parabolic subgroup, the pull-back map
$$
\CH^*(G_0/P)\to\CH^*(G_0/Q)
$$
is injective and sends $Z_w$ in $\CH^*(G_0/P)$ to $Z_w$ in $\CH^*(G_0/Q)$. Sometimes we write
$Z_{[i_1,\ldots,i_l]}$ for $Z_w$ with $w=s_{i_1}\cdots s_{i_l}$ a reduced decomposition.
\end{ntt}

\begin{rem}[Comparison with other algorithms] \label{Steenrod.rem}
There are many recipes in the literature for computing the multiplication
table in the basis $Z_w$, a.k.a.~the generalized Littlewood-Richardson coefficients,
see e.g.~\cite{De74}.
But as far as the authors know, only the one in \cite{DuZ07} and \cite{Duan} can be adapted to computing also the Steenrod operations.  
It is based on the consideration
of the Bott-Samelson resolution of $G_0/B$.
This resolution is a toric variety, and
the structure of its Chow ring and the structure of its Steenrod algebra
are both well-known, and one finds explicit combinatorial formulas.  The algorithm presented below is in terms of equivariant cohomology, so is more general than the Duan-Zhao algorithm.  Also, our practical experience in performing the calculations used below in Lemma \ref{le7} suggests that our algorithm can be substantially faster.
\end{rem}

\begin{ntt}
Let $\widehat T$ be the group of characters of $T$, that is the root lattice of $G_0$, with basis consisting of fundamental roots $\alpha_1,\ldots,\alpha_n$.  (In particular, $n$ is the rank of $G$.) The ring $\CH^*_T(\pt)$ coincides with the symmetric algebra
$S(\widehat T)\simeq\zz[\alpha_1,\ldots,\alpha_n]$ of $\widehat T$.

Observe that the pullback of the structural map gives $\CH^*_T(G_0/P)$ the structure of a $\CH^*_T(\pt)$-module.
We have
$$
\widehat T\subset\CH^*_T(\pt)\to\CH^*_T(G_0/P)\twoheadrightarrow\CH^*(G_0/P),
$$
the last map is surjective with kernel generated by the image of $\widehat T$, see \cite[Section~2.3]{Bri97}. 
\end{ntt}

\begin{ntt}
There are $T$-equivariant analogs of Schubert classes $Z^T_w$ whose images
in $\CH^*(G_0/P)$ are $Z_w$.
The $T$-fixed points of $G_0/P$ are parametrized by $W/W_P$, see \cite[Section~2]{GHZ}.
Let $\ii_w$ be the pull-back map
$$
\CH^*_T(G_0/P)\to\CH^*_T(\pt)
$$
induced by the inclusion of the fixed point corresponding to $w\in W/W_P$.
Then the direct sum map
$$
\CH^*_T(G_0/P)
\xhookrightarrow{\oplus_{w\in W/W_P}\ii_w}\displaystyle\bigoplus_{w\in W/W_P}\CH^*_T(\pt)
$$
is injective, see \cite[Theorem~2.1]{GHZ}.
\end{ntt}

\begin{lem}[{cf.~\cite[\S2]{KT03}}]\label{tao}
Fix $w\in W/W_P$.
\begin{enumerate}
\item\label{tao:1} $\ii_u(Z^T_w)=0$ for all $u\not\ge w$ in the strong Bruhat order.
\item\label{tao:2} $\ii_w(Z^T_w)=\displaystyle\prod_{\alpha\in\Phi^+,\,w^{-1}(\alpha)\in\Phi^-}\alpha$.
\item\label{tao:3} For any $x\in\CH^*_T(G_0/P)$ with $\ii_u(x)=0$ for all $u\not\ge w$,
the polynomial $\ii_w(x)$ is divisible by $\ii_w(Z^T_w)$ in $\CH_T^*(\pt)$.
\end{enumerate}
\end{lem}

Actually $\ii_v(Z_w^T)$ can be computed via the \emph{generalized Billey formula} (\cite[Theorem~7.1]{Ty}).
Namely, if $v=s_{i_1}\ldots s_{i_l}$ is a reduced decomposition, we have
\begin{equation}\label{eq:Billey}
\ii_v(Z_w^T)=\sum_{s_{i_{j_1}}\ldots s_{i_{j_k}}\in R(w)}r(j_1)\ldots r(j_k),
\end{equation}
where $r(j)=s_{i_1}\ldots s_{i_{j-1}}(\alpha_{i_j})$ and $R(w)$ is the set of all reduced
decompositions of $w$. Properties \eqref{tao:1} and \eqref{tao:2} immediately follow, and
\eqref{tao:3} follows from \eqref{tao:2} and \cite[(2.12)]{GHZ}.

Now we describe an algorithm to compute the ring structure and the action
of the Steenrod algebra on $\Ch^*(G_0/P)$.

\subsection*{Elimination procedure}

The core of the algorithm is the following \emph{elimination procedure} which takes as input $x\in\CH^m_T(G_0/P)$ and returns  $a_w\in\CH^*_T(\pt)$ for $w \in W/W_P$ such that $x=\sum_{w} a_wZ_w^T$.

\begin{ntt}\label{elim}
Assume we are given $\ii_w(x)$ for all $w\in W/W_P$ with $l(w)\le m$.

Extend the Bruhat order to a linear order on $W/W_P$. We remark that the elimination procedure (but not the
final result) formally depends on the extension of the Bruhat order.
Let $u\in W/W_P$ be the minimal element such that $\ii_u(x)\ne 0$.
If such $u$ does not exist, then $x=0$.

Then by Lemma~\ref{tao}\eqref{tao:3} $\ii_u(x)$ is divisible by $\ii_u(Z_u^T)$. In particular,
since $\deg(\ii_u(x))=m$ and $\deg(\ii_u(Z_u^T))=l(u)$ by Lemma~\ref{tao}\eqref{tao:2}, we have
$l(u)\le m$.

Assume $l(u)<m$. Using the explicit formula of Lemma~\ref{tao}\eqref{tao:2} we compute the quotient polynomial
$b_u:=\frac{\ii_u(x)}{\ii_u(Z_u^T)}$ and set $x':=x-b_u\cdot Z_u^T$. Now we apply the same
procedure to $x'$ instead of $x$. Observe that by construction $\ii_u(x')=0$
and $\ii_w(x')=0$ for $w<u$ by Lemma~\ref{tao}\eqref{tao:1}.
Therefore eventually we will arrive to the situation when either $x=0$ or
$l(u)=m$.

Consider now all $v\in W/W_P$ such that $l(v)=m$. Since these $v$'s have the same
length, they are incomparable in the Bruhat order. Therefore the same consideration
shows that $b_v:=\frac{\ii_v(x)}{\ii_v(Z_v^T)}$ are integers. Set
$y:=x-\sum_{l(v)=m}b_vZ_v^T$.

We claim that $y=0$. Indeed, assume $y\ne 0$. Let $u$ be the minimal element
such that $\ii_u(y)\ne 0$. Then again $\ii_u(y)$ is divisible by $\ii_u(Z_u^T)$.
But $\deg(\ii_u(y))=m$ and $\deg(\ii_u(Z_u^T))=l(u)>m$. This finishes our
elimination procedure.
\end{ntt}

\subsection*{Multiplication, Steenrod operations}

\begin{nttmult}
Let $u,v\in W/W_P$. Using the elimination procedure, we compute the expansion
\[
Z_u^T\cdot Z_v^T=\sum_{w\in W/W_P}a_wZ_w^T \quad \text{with $a_w\in\CH^*_T(\pt)$.}
\]
Therefore $$Z_u\cdot Z_v=\sum_{w\in W/W_P}\bar a_wZ_w  \quad \text{in $\CH^*(G_0/P)$},$$
where $\bar a_w$ is the image of $a_w$ under the homomorphism
$$\CH^*_T(\pt)\to\CH^*(\pt)=\zz$$
that sends a polynomial to its constant term.
\end{nttmult}

\begin{nttsteen}
Let $p$ be a prime number and assume $\car k\ne p$. For $x\in\Ch^*(G_0/P)$
let $$S^\bullet(x)=\sum_{j\ge 0}S^j(x)t^j\in\Ch^*(G_0/P)[t]$$
denote the total Steenrod operation (see e.g. \cite{Br03}).

Recall that $\CH^*_T(\pt)$ is the polynomial ring $\zz[\alpha_1,\ldots,\alpha_n]$
in simple roots.
Set $\Ch^*_T(\pt):=\CH^*_T(\pt)/p$.

The total Steenrod operation on
$\Ch^*_T(\pt)$ is given by
\[
\begin{array}{rcl}
\ff_p[\alpha_1,\ldots,\alpha_n]&\xrightarrow{S^\bullet}&\ff_p[\alpha_1,\ldots,\alpha_n][t]\\
\alpha_j&\mapsto& \alpha_j+t\alpha_j^p.
\end{array}
\]

Let $u\in W/W_P$, $j\ge 0$. Using the elimination procedure, we find the expansion
\[
S^j(Z_u^T)=\sum_{v\in W/W_P}a_vZ_v^T \quad \text{with
$a_v\in\Ch^*_T(\pt)$.}
\]
Then $$S^j(Z_u)=\sum_{v\in W/W_P}\bar a_v Z_v.$$
\end{nttsteen}

\begin{nttchern}
There is an effective procedure to compute the $T$-equivariant Chern classes
of a $G_0$-equivariant vector bundle $V$ on $G_0/P$.
Indeed, the $r$-th Chern class $c_r^T(V)$ of $V$ is the image of
the Chern class $c_r^{G_0}(V)$ under the map
$$
\CH^*_{G_0}(G_0/P)\to\CH^*_T(G_0/P),
$$
so it lies inside $\CH^*_T(G_0/P)^W$ (the Weyl group action as in \cite[Section~4]{Ty}), and its value is determined by one entry, namely,
$$
\ii_w(c_r^T(V))=w(\ii_{[]}(c_r^T(V))).
$$
The entry at $[]$ can be computed through the map
$$
\CH^*_{G_0}(G_0/P)\to\CH^*_{P}(\pt)\to\CH^*_T(\pt),
$$
and it coincides with the $r$-th elementary symmetric function in roots of $V$.
Then one applies the elimination procedure. We now illustrate this.
\end{nttchern}

\begin{ex}\label{exg2}
Let $G_0$ be the split group of type $\mathrm{G}_2$ and $P$ its parabolic subgroup of type $2$.
There are exactly two $5$-dimensional twisted $G_0$-homogeneous flag varieties: a projective quadric,
which is the variety of parabolic subgroups of type $1$, and $G_0/P$ (which is not a quadric).
We compute some products in $\CH^*(G_0/P)$.

The representatives of minimal length in $W/W_P$ in the
(decreasing) Bruhat order are:
$$Z_{[2,1,2,1,2]},Z_{[1,2,1,2]},Z_{[2,1,2]},Z_{[1,2]},Z_{[2]},Z_{[]}.$$
                                                                                     
Put $\ii :=\oplus_{w\in W/W_P}\ii_w$.

Let us compute, say, $\ii_{[2,1,2,1,2]}(Z_{[2]}^T)$. By \eqref{eq:Billey} we have
$$
\ii_{[2,1,2,1,2]}(Z_{[2]}^T)=r(1)+r(3)+r(5)=\alpha_2+s_1s_2(\alpha_1)+s_2s_1s_2s_1(\alpha_2)
=6\alpha_1+4\alpha_2
$$
Computing the other entries in the same way we get
\begin{equation}\label{eq1}
\ii(Z_{[2]}^T)=(6\alpha_1+4\alpha_2,\quad 6\alpha_1+3\alpha_2,\quad 3\alpha_1+3\alpha_2,
\quad 3\alpha_1+\alpha_2,\quad \alpha_2,\quad 0).
\end{equation}

Let us compute $(Z_{[2]}^T)^2$.
Squaring~\eqref{eq1} we obtain:
$$
\ii((Z_{[2]}^T)^2)=((6\alpha_1+4\alpha_2)^2,\quad (6\alpha_1+3\alpha_2)^2,\quad (3\alpha_1+3\alpha_2)^2,
\quad (3\alpha_1+\alpha_2)^2,\quad \alpha_2^2,\quad 0).
$$
Applying the elimination procedure we get
$$
\ii_{[1,2]}((Z_{[2]}^T)^2-\alpha_2 Z_{[2]}^T)=3\alpha_1(\alpha_1+\alpha_2),
$$
but
$$
\ii_{[1,2]}(Z_{[1,2]}^T)=\alpha_1(\alpha_1+\alpha_2),
$$
so
$$
(Z_{[2]}^T)^2=\alpha_2 Z_{[2]}^T+3Z_{[1,2]}^T
$$
and, in particular, $Z_{[2]}^2=3Z_{[1,2]}$.

Continuing this way we can recover the whole multiplication table in $\CH^*_T(G_0/P)$ and $\CH^*(G_0/P)$.

Let us compute $S^1(Z_{[1,2,1,2]}^T)$ for $p=2$ now. The generalized Billey formula~\eqref{eq:Billey}
gives
\begin{align*}
\ii(Z_{[1,2,1,2]}^T)=(&(\alpha_1+\alpha_2)(3\alpha_1+2\alpha_2)(2\alpha_1+\alpha_2)(3\alpha_1+\alpha_2),\\
&\alpha_1(3\alpha_1+2\alpha_2)(2\alpha_1+\alpha_2)(3\alpha_1+\alpha_2),\quad 0,\quad 0,\quad 0,\quad 0)\\
=(&(\alpha_1^2+\alpha_2^2)\alpha_1\alpha_2,\quad \alpha_1^2\alpha_2(\alpha_1+\alpha_2),\quad 0,\quad 0,\quad 0,\quad 0).
\end{align*}

Substituting $\alpha_1\mapsto\alpha_1+t\alpha_1^2$, $\alpha_2\mapsto\alpha_2+t\alpha_2^2$, taking modulo $2$ and taking the coefficient
at $t$ we get:
$$S^1(\ii(Z_{[1,2,1,2]}^T))=((\alpha_1^2+\alpha_2^2)(\alpha_1^2\alpha_2+\alpha_1\alpha_2^2),\quad \alpha_1^3\alpha_2(\alpha_1+\alpha_2),\quad 0,\quad 0,\quad 0,\quad 0)$$
and the elimination procedure gives $$S^1(Z_{[1,2,1,2]}^T)=\alpha_1 Z_{[1,2,1,2]}^T+Z_{[2,1,2,1,2]}^T.$$
In particular, $S^1(Z_{[1,2,1,2]})=Z_{[2,1,2,1,2]}=\pt$.

Now we compute the second Chern class $c^T_2$ of the tangent bundle of $G_0/P$.
The roots of this bundle are:
$$\alpha_2,\quad \alpha_1+\alpha_2,\quad 3\alpha_1+2\alpha_2,\quad 2\alpha_1+\alpha_2,\quad 3\alpha_1+\alpha_2.$$
The total Chern class equals
$$(1+t\alpha_2)(1+t(\alpha_1+\alpha_2))(1+t(3\alpha_1+2\alpha_2))(1+t(2\alpha_1+\alpha_2))(1+t(3\alpha_1+\alpha_2)).
$$
The coefficient at $t^2$ is $29\alpha_1^2+14\alpha_2^2+42\alpha_1\alpha_2$.
Now
\begin{align*}
&s_2(29\alpha_1^2+14\alpha_2^2+42\alpha_1\alpha_2)=29\alpha_1^2+\alpha_2^2+16\alpha_1\alpha_2,\\
&s_1(29\alpha_1^2+\alpha_2^2+16\alpha_1\alpha_2)=-10\alpha_1^2+\alpha_2^2-10\alpha_1\alpha_2,\\
&s_2(-10\alpha_1^2+\alpha_2^2-10\alpha_1\alpha_2)=-10\alpha_1^2+\alpha_2^2-10\alpha_1\alpha_2,\\
&s_1(-10\alpha_1^2+\alpha_2^2-10\alpha_1\alpha_2)=29\alpha_1^2+\alpha_2^2+16\alpha_1\alpha_2,\\
&s_2(29\alpha_1^2+\alpha_2^2+16\alpha_1\alpha_2)=29\alpha_1^2+14\alpha_2^2+42\alpha_1\alpha_2,
\end{align*}
so we have
\begin{align*}
\ii(c_2^T)=(&29\alpha_1^2+14\alpha_2^2+42\alpha_1\alpha_2,\quad 29\alpha_1^2+\alpha_2^2+16\alpha_1\alpha_2,\quad
-10\alpha_1^2+\alpha_2^2-10\alpha_1\alpha_2,\\
&-10\alpha_1^2+\alpha_2^2-10\alpha_1\alpha_2,\quad
29\alpha_1^2+\alpha_2^2+16\alpha_1\alpha_2,\quad 29\alpha_1^2+14\alpha_2^2+42\alpha_1\alpha_2).
\end{align*}

By the elimination procedure we obtain:
$$c^T_2=(29\alpha_1^2+14\alpha_2^2+42\alpha_1\alpha_2)Z_{[]}^T-13(2\alpha_1+\alpha_2)Z_{[2]}^T+13Z_{[1,2]}^T.$$
In particular, the ordinary second Chern class of the bundle equals $13Z_{[1,2]}$.
\end{ex}

\section{Chernousov-Merkurjev formula}
Recall that $G$ denotes a semisimple algebraic group of inner type. Let
$X$ and $X'$ be twisted $G$-homogeneous flag varieties. We present $G$ as a twisted form of a
split group $G_0$. Then $X$ and $X'$ are twisted forms of $G_0/P$ and $G_0/P'$
resp. for some standard parabolic subgroups $P$, $P'$ of $G_0$. We say that $X$ and
$X'$ are homogeneous varieties of type $P$ and $P'$ respectively.

In \cite[Proposition~13]{CMe06} Chernousov and Merkurjev construct a filtration
by closed subvarieties on $X\times X'$ such that the successive differences
are affine fibrations over $Y_w$ of rank $l(W_PwW_{P'})$, where $w$ runs over
the representatives of $W_P\backslash W/W_{P'}$, $W$, $W_P$, $W_{P'}$ are the Weyl groups of $G_0$,
$P$, $P'$ resp., $l(W_PwW_{P'})$ is the length of the minimal representative of the
double coset $W_PwW_{P'}$, and $Y_w$ is a twisted form of $G_0/Q_w$ with
$Q_w=R_uP\cdot(P\cap wP'w^{-1})$, where $R_uP$ stands for the unipotent radical
of $P$. Note that by \cite[Lemma~7]{CMe06} $Q_w$ is a standard parabolic
subgroup of $G_0$ and is contained in $P$.

By \cite[Theorem~4.4]{NeZ} we get:
\begin{prop}\label{lchme}
In the above notation
$$
\CH^*(X\times X')\simeq\bigoplus_{w\in W_P\backslash W/W_{P'}}\CH^{*-\dim X-\dim X'+\dim Y_w+l(W_PwW_{P'})}(Y_w).
$$
\end{prop}               

\begin{rem}
Theorem~4.4 in \cite{NeZ} is proved for any oriented cohomology theory
using resolution of singularities. So, it is assumed there that $\car k=0$.
For Chow groups this is not necessary, cf.~\cite[Theorem~66.2]{EKM}.
\end{rem}

\begin{ex}\label{62}
If $G$ is a special orthogonal group, and $X=X'=X_1$ is a projective quadric
of dimension at least $5$, then
$$\CH^*(X\times X)\simeq\CH^*(X)\oplus\CH^{*-1}(X_{1,2})\oplus\CH^{*-\dim X}(X).$$
\end{ex}

We now develop an important tool to produce rational projectors.

Let $w\in W_P\backslash W/W_{P'}$ and $f\colon\BY_w\to\BX$ be the natural map induced by the inclusion $Q_w\subset P$.

For a Schubert cycle $\beta=[\overline{BuQ_w/Q_w}]\in\CH^*(\BY_w)$,
$u\in W$, define an element $\beta_\star(1)\in\CH^*(\overline{X'})$ as
$$
\beta_\star(1)=\begin{cases}
[\overline{BuwP'/P'}], & \text{if }l(uwW_{P'})=l(uW_{Q_w})+l(W_PwW_{P'});\\
0, & \text{otherwise,}
\end{cases}
$$
and for an arbitrary $\beta\in\CH^*(\BY_w)$ define $\beta_\star(1)$ by linearity.

Fix a rational cycle $\alpha\in\CH^*(\BY_w)$ and
for a cycle $x\in\CH^*(\BX)$ define $$\alpha_\star(x)=(\alpha\cdot f^*(x))_\star(1)\in\CH^*(\overline{X'}).$$ 

\begin{thm}\label{tchme}
In the preceding notation, for $\alpha\in\CH^*(\BY_w)$
$$\alpha_\star\colon\CH(\BX)\to\CH(\overline{X'})$$ is the realization of a rational
cycle on $\BX\times\overline{X'}$. Moreover, the realization of any rational cycle on
$\BX\times\overline{X'}$ can be constructed in this way.

In particular, if $P=P'$, then
for $\alpha\in\Ch^{\dim Y_w-\dim X+l(W_PwW_P)}(\BY_w)$
some power of $\alpha_\star$ is the realization
of a rational projector on $\BX$, and
the realization of any rational projector on $\BX$ can be constructed in this way.
\end{thm}
\begin{proof}
Let $X$ and $X'$ be homogeneous $G$-varieties of type $P$ and $P'$. Consider the following diagram
$$
\xymatrix{
G_0/(P\cap wP'w^{-1})\ar@{^{(}->}@/^1pc/[rr]^-{j}\ar@{^{(}->}[r]\ar[d]^-{\pi'}&\Gamma_w\ar[r]_-{i}\ar[ld]^-{\pi}&G_0/P\times G_0/P'\ar[r]^-{\pr_2}\ar[ld]^-{\pr_1}&G_0/P'\\
G_0/Q_w\ar[r]^-{f}&G_0/P
}
$$
where the maps $\pi'$ and $f$ are induced by inclusions $P\cap wP'w^{-1}\subset Q_w\subset P$,
$G_0/(P\cap wP'w^{-1})$ is considered as a subvariety of $G_0/P\times G_0/P'$
under the map
$$j\colon g(P\cap wP'w^{-1})\mapsto (gP,gwP'),\quad g\in G_0,$$
$\Gamma_w$ is the closure in $G_0/P\times G_0/P'\times G_0/Q_w$ of the image
of the graph of $\pi'$ under the map $j\times\id$,
and $i$ and $\pi$ are induced by the projections.

The proof in \cite{NeZ} shows that the image of an element $\alpha\in\CH^*(G_0/Q_w)$
under the isomorphism of Proposition~\ref{lchme} equals $i_*\pi^*(\alpha)$.
Further, we identify the image of $\alpha$ with its realization, i.e., with the homomorphism
\begin{align*}
\alpha_\star\colon\CH^*(G_0/P)&\to\CH^*(G_0/P')\\
x&\mapsto(\pr_2)_*(i_*\pi^*(\alpha)\cdot\pr_1^*(x)).
\end{align*}
The above diagram and the projection formula show that
\begin{align}\label{ff1}
\nonumber
\alpha_\star(x) &=(\pr_2)_*(i_*\pi^*(\alpha)\cdot\pr^*_1(x))=(\pr_2)_*(i_*(\pi^*(\alpha)\cdot i^*\pr_1^*(x)))\\
&=(\pr_2)_*(i_*\pi^*(\alpha\cdot f^*(x)))=(\alpha\cdot f^*(x))_\star(1).
\end{align}
In particular, to compute $\alpha_\star(x)$, we just need to know the image of $\beta_\star(1)$ for each element
$\beta\in\CH^*(G_0/Q_w)$. One sees directly that for a Schubert cycle $\beta=[\overline{BuQ_w/Q_w}]$ 
                
\begin{equation}\label{ff2}
\beta_\star(1)=\begin{cases}
[\overline{BuwP'/P'}], & \text{if }l(uwW_{P'})=l(uW_{Q_w})+l(W_PwW_{P'});\\
0, & \text{otherwise.}
\end{cases}
\end{equation}

To finish the proof of the theorem it remains to set $P'=P$ and note that in $\End(\M(\BX))$ some power of
any element is a projector.
\end{proof}

\begin{ex}\label{ex65}
Let $A$ be a central simple algebra of degree $n+1$, $G=\PGL_1(A)$,
$P=P_1$, and $P'=P_n$. Then $G_0=\PGL_{n+1}$, $X$ is the Severi-Brauer variety $\SB(A)$, $X'=\SB(A^{\op})$,
$W=\Sym(\{1,\ldots,n+1\})$ is the symmetric group on letters $1,\ldots,n+1$,
$W_P=\Sym(\{2,\ldots,n+1\})$, and $W_{P'}=\Sym(\{1,\ldots,n\})$.

The minimal representatives of the cosets $W/W_{P'}$ are
$$\{1,\,s_n,\,s_{n-1}s_n,\,\ldots,\,s_1\cdots s_{n-1}s_n\},$$ where $s_i=(i,i+1)$ is a
simple transposition, and the minimal representatives of the cosets
$W/W_P$ are $$\{1,\,s_1,\,s_2s_1,\,\ldots,\,s_n\cdots s_2s_1\}.$$
The minimal representatives of the double cosets
$W_P\backslash W/W_{P'}$ are $\{1,s_1\cdots s_{n-1}s_n\}$. Take $w=1$. Then
$Q_w=P_{1,n}$.

By Proposition~\ref{lchme} we have
$$\CH^*(\SB(A)\times\SB(A^\op))\simeq\CH^{*-1}(\SB_{1,n}(A))\oplus\CH^*(\SB(A)),$$
where $\SB_{1,n}(A)$ is the variety of parabolic subgroups of type $P_{1,n}$
(known also as the incidence variety).

Take (a rational) $\alpha=1\in\CH^0(G_0/P_{1,n})$. Let $h_1$ and $h_n$ be
the Schubert cycles in $\CH^1(G_0/P)$ and $\CH^1(G_0/P')$. All Schubert cycles
in $\CH^*(G_0/P)$ equal $1,h_1,\ldots,h_1^n$. We compute $\alpha_\star(h_1^i)$ now.

We have $\alpha_\star(h_1^i)=(f^*(h_1^i))_\star(1).$ The cycle $h_1^i$ equals
$[\overline{Bv_iP/P}]$ with $v_i=s_{n-i}\cdots s_1$ and
$f^*(h_1^i)=[\overline{Bv_iv_0P_{1,n}/P_{1,n}}]$ with $v_0=s_2\cdots s_n$.
By formula~\eqref{ff2} $$(f^*(h_1^i))_\star(1)=\begin{cases}
h_n, &\text{if } i=n;\\
1, &\text{if } i=n-1;\\
0, &\text{otherwise}.
\end{cases}$$
Thus $\alpha$ as an element in $\CH^*(G_0/P\times G_0/P')$ equals
$h_1\times 1+1\times h_n$. So, the latter cycle is rational.
\end{ex}

\section{Weak special correspondences}
\begin{dfn}
Let $p$ be a prime number, and $X$ be a smooth projective irreducible variety over $k$
of dimension $b(p-1)$ for some $b$.
A cycle $\rho\in\Ch^b(X\times X)$ is called a \emph{weak special correspondence}, if
$\rho_{k(X)}=H\times 1-1\times H$ for some $H\in\Ch^b(X_{k(X)})$,
$\bar\pi:=c\cdot\rho^{p-1}_{k(X)}$ is a projector for some $c\in\ff_p^\times$, and
$$(X_{k(X)},\bar\pi)\simeq\bigoplus_{i=0}^{p-1}\ff_p(bi).$$
\end{dfn}

\begin{lem}[Rost, {\cite[Section~9]{Ro07}}]\label{le62}
Assume that $X$ possesses a weak special correspondence, has no zero-cycles
of degree coprime to $p$, and $\car k=0$. Then $\dim X=p^n-1$ for some $n$.
\end{lem}

\begin{lem}
Assume that $p\in\{2,3\}$.
Let $X$ be a smooth projective irreducible variety over $k$ of dimension $b(p-1)$
with no zero-cycles of degree coprime to $p$, and $\pi$ a projector over $k$ such that
$(X_{k(X)},\pi_{k(X)})\simeq\bigoplus_{i=0}^{p-1}\ff_p(bi)$. Then $X$ possesses a weak special correspondence.
\end{lem}
\begin{proof}
Denote $\bar\pi=\pi_{k(X)}$ and $\BX=X_{k(X)}$. Since $(\BX,\bar\pi)\simeq\bigoplus_{i=0}^{p-1}\ff_p(bi)$,
the projector $\bar\pi$ equals $\sum_{i=0}^{p-1} h_i\times g_i$ for some $h_i\in\Ch^{bi}(\BX)$ and
$g_i\in\Ch_{bi}(\BX)$ with $\deg(h_ig_i)=1$ for all $i$.

Note first that $\pi^t\circ\pi$ contains at most $p$ summands and is non-zero, since
$$(g_{p-1}\times h_{p-1})\circ(h_0\times g_0)=d\cdot h_0\times h_{p-1}\ne 0,$$
where $d=\deg(g_0g_{p-1})\in\ff_p^\times$. Therefore,
since $X$ has no zero-cycles of degree coprime to $p$, we can assume that $g_i=h_{p-1-i}$
for all $i$. In particular, this proves our lemma for $p=2$.

Write $f\colon\Spec k(X)\to X$ for the generic point.
By the generic point diagram (see \cite[Lemma~1.8]{PSZ}) there is a cycle $\alpha\in\Ch^b(\BX\times\BX)$
such that $\beta:=h_1\times 1+\alpha$ is defined over $k$ and $(\id_X^*\times f)(\alpha)=0$.

Consider $\bar\pi\circ\beta\circ\bar\pi$. A direct computation shows that this cycle
equals $\rho_1:=h_1\times 1+a_11\times h_1$ for some $a_1\in\ff_p$. By symmetry we can assume that $a_1\ne 0$.
If $p=3$, then set $c=\deg(h_1^2)^{-1}\in\ff_p^\times$. 
The cycle $\rho_1^2=h_1^2\times 1-a_1 h_1\times h_1+1\times h_1^2$. Since $X$ has no zero-cycles
of degree coprime to $p$, we have $a_1=-1$. Moreover, $c\cdot\rho_1^2$ is a projector. Thus,
$\rho_1$ is a weak special correspondence on $X$.
\end{proof}

\begin{rem}
Using messier computations one can also prove the above Lemma for $p=5$.
\end{rem}

\begin{lem}
Let $X$ be a smooth projective irreducible variety over $k$ with $\car k=0$ and $M$ a direct summand of its motive.
Assume that $M$ is indecomposable and generically split
and $M_{k(X)}\simeq\bigoplus_{i\in I\cup\{0\}}\ff_p(i)$
for some multiset of positive indexes $I$.

Then there exists a smooth projective irreducible variety $Y$ over $k$
such that $M$ is isomorphic to an upper direct summand of $\M(Y)$ and $\dim M=\dim Y$.
\end{lem}
\begin{proof}
Let $Y'$ be a closed irreducible subvariety of $X$ of minimal dimension with respect
to the property that $Y'_{k(X)}$ has a zero-cycle of degree coprime to $p$.

By \cite[Lemma~7.1]{Sem09} there exists a smooth projective irreducible variety
$\widetilde{Y'}$ birational to $Y'$ such that both $\widetilde{Y'}_{k(X)}$
and $X_{k(\widetilde{Y'})}$ have zero-cycles of degree coprime to $p$.
Since the upper motive $M$ of $X$ is generically split, Rost nilpotence
holds for its endomorphism ring, i.e., the kernel of the natural map
$\End(M)\to\End(\overline M)$ consists of nilpotent correspondences by \cite[Prop.~3.1]{ViZ}.
Therefore $M$ is also an upper direct summand of $\widetilde{Y'}$.
Hence, $\dim Y'=\dim \widetilde Y'\ge\dim M$.

Let now $Y''$ be the generic point of $M$ (see Lemma~\ref{genpt2}).
Obviously, $Y''_{k(X)}$ in not $0$ in $\Ch(X_{k(X)})$, and therefore without loss of generality
we can assume that $Y''$ is represented by a closed subvariety of $X$, which we denote
by the same letter.
By \cite[Remark~5.6]{KM06} the variety $Y''$ has the property that $Y''_{k(X)}$
has a zero-cycle of degree coprime to $p$. Since $\dim Y''=\dim M$, the dimension
of $Y''$ is minimal with respect to this property.

Therefore by \cite[Lemma~7.1]{Sem09} there exists a smooth projective irreducible variety $Y$ birational
to $Y''$ with required properties.
\end{proof}

The following statement for $p=2$ might be called a {\it binary motive theorem}.

\begin{cor}\label{cor1}
Assume that $p\in\{2,3\}$ and $\car k=0$.
Let $X$ be a smooth projective irreducible variety with no zero-cycles of degree coprime to $p$
and $M$ a direct summand of $\M(X)$. If
$M_{k(X)}\simeq\bigoplus_{i=0}^{p-1}\ff_p(bi)$ for some integer $b$, then $\dim M=p^n-1$ for some $n$.$\hfill\qed$
\end{cor}

\begin{prop}\label{prpr}
Let $G$ be a split semisimple algebraic group of inner type over a field $k$ with $\car k=0$ and
$\xi\in H^1(k,G)$. Let $p\in\{2,3\}$. Consider a twisted $_\xi G$-homogeneous flag variety $X$
and write $$\M(X_{k(X)})\simeq\oplus_{i\in I}\ff_p(i)\bigoplus\oplus_{j\in J}N_j$$
with indecomposable direct summands $N_j$ of positive dimension.

Assume that the following conditions hold:
\begin{enumerate}
\item For all $j$ the motives $N_j$ are defined over $k$.
Moreover, there exist twisted flag varieties $Y_j$ over $k$ such that $N_j=\U(Y_j)$ and
every cycle in $\Ch(\overline{Y_j\times X})$ which is defined over $k(Y_j)(X_{k(Y_j)})$ is defined over $k(Y_j)$.
\item The variety $X$ has no zero-cycles of degree coprime to $p$.
\item \label{prpr.3} Let $Q(t)$ denote the Poincar\'e polynomial of the (graded by codimension)
subgroup of $\Ch^*(\BX)$ generated by the rational cycles of the first shell.
Assume $$\frac{\sum_{i\in I}t^i}{Q(t)}=\sum_{l=0}^{p-1}t^{bl}$$ for some $b$.
\end{enumerate}
Then $b=\frac{p^n-1}{p-1}$ for some integer $n$.

\end{prop}
\begin{proof}
Since by assumption the motives $N_j$ are defined over $k$, we use for simplicity the same notation $N_j$ over $k$ and over $k(X)$.

Since $N_j$ are defined over $k$, are indecomposable over $k(X)$ and have positive dimension,
we can apply Proposition~\ref{lekarp2}. So, 
\[
\M(X)\simeq U\oplus\bigoplus_{j\in J}N_j
\] over $k$, where $U$ is
a motive with Poincar\'e polynomial $\sum_{i\in I}t^i$, since by our assumptions
$U_{k(X)}\simeq\oplus_{i\in I}\ff_p(i)$.

It follows from Theorem~\ref{t1} that $U\simeq\oplus_{s\in S} M(s)$ for some motive
$M$ and $Q(t)=\sum_{s\in S}t^s$. In particular, by assumption \eqref{prpr.3}, $P(M,t)=\sum_{l=0}^{p-1}t^{bl}$.
The proposition follows now from Corollary~\ref{cor1}.
\end{proof}

\section{Applications to motives of twisted flag varieties: type $\E_6$}\label{s8}

The goal of this section is to provide a complete classification of all
possible motivic decompositions of twisted $G$-homogeneous flag varieties for $G$
a group of inner type $\E_6$.
Note that with $\ff_p$-coefficients and $p \ne 2,3$,
every twisted $G$-homogeneous flag variety is a direct sum of Tate motives, and
the case $p=2$ was settled in \cite[p.~1048]{PSZ}. Therefore we only consider $\ff_3$-coefficients here.
All decomposition types are collected in Table \ref{e6.dec}.  

Throughout we will refer to the \emph{Tits algebra $A$ of $G$}, by which we mean a Tits
algebra for the vertex $1$ in the sense of \cite[6.4.1]{Ti71}. (This is a special case of the more general theory of Tits algebras from
\cite{Ti71} or \cite[\S27]{KMRT}.) This algebra $A$ is a central simple algebra of degree $27$ and
is determined up to isomorphism or anti-isomorphism by $G$. By $D$ we denote the underlying central simple division
algebra.

\begin{table}[hbtp]
\begin{center}
\begin{tabular}{c|r|c}
$J_3(G)$ & $\Theta$ & $\M(X_\Theta)$ \\
\hline
$(2,1)$ & $2$ & $M_{2,1}\oplus M_{2,1}(1)$\\
 & $4$ & $M_{2,1}\oplus(\oplus_{j\in J^{2,1}} R_{2,1}(j))\oplus M_{2,1}(9)$ \\
 & $\{2,4\}$ & $\M(X_4)\oplus \M(X_4)(1)$\\
 & any other & $\bigoplus_{i\in I^{2,1}_\Theta}R_{2,1}(i)$\\
\hline
$(1,1)$ & $2$ & $M_{1,1}\oplus(\oplus_{i=4}^7R_{1,1}(i))\oplus M_{1,1}(1)$\\
 & $4$ & $M_{1,1}\oplus(\oplus_{j\in J^{1,1}} R_{1,1}(j))\oplus M_{1,1}(9)$\\
 & $\{2,4\}$ & $\M(X_4)\oplus \M(X_4)(1)$\\
 & any other & $\bigoplus_{i\in I^{1,1}_\Theta}R_{1,1}(i)$\\
\hline
$(0,1)$ & any & $\bigoplus_{i\in I^{0,1}_\Theta} R_{0,1}(i)$ \\
\hline
$(1,0)$ & $2$ & $\bigoplus_{i=0,1,10,11,20,21}\ff_3(i)\oplus\bigoplus_{j\in
J^{1,0}_2}\M(\SB(D))(j)$ \\
 & $4$ &  $\bigoplus_{i=0,1,9,10,10,11,19,20,20,21,29,30}\ff_3(i)\oplus\bigoplus_{j\in J^{1,0}_4}\M(\SB(D))(j)$\\
 & $\{2,4\}$ & $\M(X_4)\oplus \M(X_4)(1)$\\
 & any other & $\bigoplus_{i\in I^{1,0}_\Theta}\M(\SB(D))(i)$\\
\hline
$(0,0)$ & any & $\oplus_{i\in I^{0,0}_\Theta}\ff_3(i)$
\end{tabular}
\caption{Motivic decomposition of twisted flag varieties of $\E_6$ mod $3$} \label{e6.dec}
\end{center}
\end{table}

\begin{table}[hbtp]
\begin{center}
\begin{tabular}{c|l}
Motive & Poincar\'e polynomial \\
\hline
$M_{2,1}$ & $\frac{(t^4+1)(t^{12}-1)(t^6+t^3+1)}{t^2-1}$\\
$M_{1,1}$ & ${\scriptstyle t^{20}+t^{18}+t^{17}+t^{16}+t^{14}+t^{13}+
t^{12}+t^{11}+2t^{10}+t^9+t^8+t^7+t^6+t^4+t^3+t^2+1}$\\
$R_{j_1,j_2}$ & $\frac{(t^{3^{j_1}}-1)(t^{4\cdot 3^{j_2}}-1)}{(t-1)(t^4-1)}$
\end{tabular}
\caption{Poincar\'e polynomials of some motives from Table \ref{e6.dec}} \label{e6.poin}
\end{center}
\end{table}

\begin{table}[hbtp]
\begin{center}
\begin{tabular}{c|l}
Multiset of indexes & Polynomial \\
\hline
$I_\Theta^{j_1,j_2}$ &$\frac{P(X_\Theta,t)}{P(R_{j_1,j_2},t)}$\\
$J^{j_1,1}$ & $\frac{P(X_4,t)-P(M_{j_1,1},t)(1+t^9)}{P(R_{j_1,1},t)}$\\
$J^{1,0}_2$ & $\frac{P(X_2,t)-(1+t+t^{10}+t^{11}+t^{20}+t^{21})}{1+t+t^2}$\\
$J^{1,0}_4$ & $\frac{P(X_4,t)-(1+t+t^{10}+t^{11}+t^{20}+t^{21})(1+t^9)}{1+t+t^2}$
\end{tabular}
\caption{Multisets of indexes appearing in Table \ref{e6.dec}} \label{e6.multi}
\end{center}
\end{table}

\subsection*{Left column: the $J$-invariant}
Let $G_0$ be a split semisimple algebraic group over $k$ and $p$ be a prime.
Denote $\overline G=G_0\times_k\ksep$, where $\ksep$ is a separable closure
of $k$.
It is known that $$\Ch^*(\overline G)\simeq{\ff_p}[x_1,\ldots,x_r]/(x_1^{p^{k_1}},\ldots,x_r^{p^{k_r}})$$
with $\deg x_i=d_i$ for some integers $r$, $k_i$, and $d_i$. We order the generators
so that $d_1\le\ldots\le d_r$ and fix one such isomorphism between $\Ch^*(\overline G)$
and this polynomial ring.

Let now $\xi\in Z^1(k,G_0)$ be a cocycle and consider the composite map
$$\Ch({_\xi{(G_0/B)}})\xrightarrow{\res}\Ch({_\xi{(G_0/B)}}\times_k\ksep)
\xrightarrow{\simeq}\Ch(G_0/B\times_k\ksep)\to\Ch(\overline G),$$
where $B$ is a Borel subgroup of $G_0$ defined over $k$, the first map is
the restriction map, the second map is induced by the isomorphism
$${_\xi{(G_0/B)}}\times_k\ksep\simeq G_0/B\times_k\ksep$$ given by $\xi$,
and the third map is induced by the canonical quotient map.
According to \cite[Definition~4.6]{PSZ} one can associate an invariant
$$J_p(\xi)=(j_1,\ldots,j_r)\in\zz^r$$ which measures the ``size'' of the
image of this composite map.
It does not depend on the choice of a separable closure $\ksep$.

Formally speaking, $J_p(\xi)$ is an invariant of $\xi$, not of $_\xi(G_0)$.  But if $G_0$ is simple and not of type $\D$ or $p \ne 2$, then the
degrees $d_i$ are pairwise distinct, and
it is a well-defined invariant of the twisted form $G={_\xi(G_0)}$
and we denote this invariant by $J_p(G)$.  For the excluded case where $G_0$ has type $\D$ and $p = 2$, see \cite{QSZ}.

We remark that some constraints on the $J$-invariants are classified in \cite[Table~4.13]{PSZ}.
E.g., if $G_0$ (equivalently, $G$) is adjoint of type $\E_6$ and $p=3$, then $r=2$, $d_1=1$, $d_2=4$,
$k_1=2$, $k_2=1$, $j_1\in\{0,1,2\}$, and $j_2\in\{0,1\}$.  We prove below that there are actually further constraints on  the $J$-invariant, see e.g.\ Corollary \ref{cor11}.

\subsection*{Remaining columns}
For the second column, recall that the simple roots of $\E_6$ are numbered as in the diagram
\begin{equation} \label{E6.roots}
\begin{picture}(7,1.9)
    \multiput(1,0.7)(1,0){5}{\circle*{\darkrad}}
    \put(3,1.45){\circle*{\darkrad}}

    \put(1,.7){\line(1,0){4}}
    \put(3,1.45){\line(0,-1){0.75}}
    
    \put(1,0){\makebox(0,0.4)[b]{$1$}}
    \put(2,0){\makebox(0,0.4)[b]{$3$}}
    \put(3,0){\makebox(0,0.4)[b]{$4$}}
    \put(4,0){\makebox(0,0.4)[b]{$5$}}
    \put(5,0){\makebox(0,0.4)[b]{$6$}}
    \put(3.3,1.3){\makebox(0,0.4)[b]{$2$}}
\end{picture}
\end{equation}
The motives $M_{j_1,j_2}$ and $R_{j_1,j_2}$ 
listed in the third column are indecomposable, and the latter is the upper motive
of the variety of Borel subgroups. Their Poincar\'e polynomials are given in Table \ref{e6.poin}.
The multisets of indexes $I_\Theta^{j_1,j_2}$ and $J^{j_1,j_2}$ in Table \ref{e6.dec} are defined as follows:
an integer $i$ appears in the multiset $s$ times iff $s$ is the coefficient at $t^i$
of the respective polynomial given in Table \ref{e6.multi}.

\smallskip

Each row of Table \ref{e6.dec} occurs over a suitable field for a suitable group.  The rest of this section and the next section are devoted to the proof of these tables.

By \cite[Prop.~4.2]{PS10} the Tits algebra $A$ is split iff the first slot $j_1$ in $J_3(G)$ equals $0$.
If $j_1=0$, then every projective homogeneous $G$-variety is generically split over a field extension
of degree coprime to $3$ and this case
was settled in \cite{PSZ}. This immediately gives all rows of Table~\ref{e6.dec} with $j_1=0$.

\subsection*{Picard groups and Tits algebras}
In this article we use some relations between rationality of the Picard groups of twisted flag varieties and
their Tits algebras, see \cite{MT95}.
E.g., if all Tits algebras of a group $G$ of inner type are split
algebras, then the Picard groups of all twisted flag varieties for the group $G$ are rational.

In this section, $G$ has inner type $\E_6$, and it follows from \cite{MT95} that the Picard groups of varieties
$X_2$, $X_4$, and $X_{2,4}$ are always rational.

We start now with some general observations.

\begin{lem}
Let $\Delta$ be a Dynkin diagram (not necessarily of type $\E_6$) and $\Psi \subseteq\Theta\subseteq\Delta$
two subsets of its vertices.
Assume that $X_\Theta$ has a rational point over $k(X_\Psi)$, and
$$P(X_\Theta,t)/P(X_\Psi,t)=t+1.$$
Then $\M(X_\Theta)=\M(X_\Psi)\oplus \M(X_\Psi)(1)$.
\end{lem}
\begin{proof}
Since $\Psi\subseteq\Theta$, we have a natural map $f\colon X_\Theta\to X_\Psi$.
The fibre $Z$ of $f$ over $k(X_\Psi)$ is a twisted flag
variety over $k(X_\Psi)$. By the assumptions the Poincar\'e polynomial
$P(Z,t)=P(X_\Theta,t)/P(X_\Psi,t)=t+1$, and $Z$ has a rational point.
Therefore $Z$ is isomorphic to $\mathbb{P}^1$.

Now by \cite[Lemma~3.3]{PSZ} $f$ is a locally trivial fibration with fiber $\mathbb{P}^1$.
Therefore \cite[Lemma~3.2]{PSZ} implies the claim.
\end{proof}

This lemma with $\Psi = \{ 4 \}$ and $\Theta = \{ 2, 4 \}$ and the classification
of Tits indices immediately imply all rows of Table \ref{e6.dec} for $X_{2,4}$.

\begin{lem}\label{l2}
If $X_2$ has a zero-cycle of degree coprime to 3, then $J_3(G) = (0,0)$ or $(1,0)$ and the index of $A$ is $1$ or $3$ respectively.
\end{lem}

\begin{proof}
As $J_3(G)$ is unchanged if we replace $k$ with an extension of degree coprime
to $3$ \cite[Prop.~5.18(2)]{PSZ}, we may assume that $X_2$ has a $k$-point.
By the classification of Tits indexes, $G$ is split or has semisimple
anisotropic kernel of type $2\A_2$.

In the second case $\ind A = 3$ and therefore $J_3(\PGL_1(A))=(1)$.
Thus, by \cite[Prop.~3.9(2)]{PS10} $J_3(G) = (1,0)$.
\end{proof}

\begin{lem}\label{l3}
The upper motives of $X_2$ and $X_4$ are isomorphic. If every zero-cycle on $X_2$
has degree divisible by $3$ and the Tits algebra of $G$ is not split, then the dimension
of its upper motive equals $20$.
\end{lem}
\begin{proof}
By hypothesis on $X_2$, $G$ is anisotropic and the Tits 3-indexes of $G_K$ as $K$ varies over all extensions
$K$ of $k$ are empty, all of $\Delta$, and $\{ 2, 4 \}$.  The claim on upper motives follows.  Moreover, as in Example \ref{tower}, there are (at most) two different big shells, the first shell $\SH_{\leqslant \{ 2 \}}$ and the last shell $\SH_{\leqslant\{1\}}$.

As before write $\U(X_2)$ for the upper motive of $X_2$.
An explicit computation of the decomposition of \cite[Theorem~7.5]{CGM} for $\M(X_2)$
shows that over $k(X_2)$ the motive of $X_2$ contains exactly six Tate motives:
$\ff_3$, $\ff_3(1)$, $\ff_3(10)$, $\ff_3(11)$, $\ff_3(20)$, and $\ff_3(21)$,
and, by assumption, the variety $X_2$ does not have a zero-cycle of degree coprime to $3$.
Therefore the number of Tate motives contained in $\U(X_2)$
over $k(X_2)$ is divisible by $3$.

Fix a generator $h$ of  the Picard group of $X_2$; it is unique up to sign.
This cycle is defined over $k$. Therefore, by Theorem~\ref{t1}
the motive $\U(X_2)(1)$ is a direct summand of $\M(X_2)$.
All this implies that $\dim\U(X_2)=20$.
\end{proof}

\begin{lem}\label{l4}
Let $J_3(G)=(j_1,j_2)$ with $j_1\ne 0$ and $M_{j_1,j_2}$ denote the upper motive of $X_2$.
If $J_3(G)\ne(1,0)$, then 
\begin{align*}
\M(X_2)&\simeq M_{j_1,j_2}\oplus M_{j_1,j_2}(1)\oplus
\bigoplus_{i\in I_1} R_{j_1,j_2}(i) \quad \text{and} \\
\M(X_4)&\simeq M_{j_1,j_2}\oplus M_{j_1,j_2}(9)\oplus
\bigoplus_{i\in I_2} R_{j_1,j_2}(i)
\end{align*}
for some multisets of indexes $I_1$ and $I_2$ (depending on $j_1, j_2$).
\end{lem} 
\begin{proof}
The formula for $X_2$ immediately follows from the proof of Lemma~\ref{l3}
and from Karpenko's theorem.

Consider now $X_4$. An explicit computation of the decomposition of \cite[Theorem~7.5]{CGM} for $\M(X_4)$
shows that over $k(X_4)$ its motive contains exactly $6$ Tate motives:
$\ff_3$, $\ff_3(9)$, $\ff_3(10)$, $\ff_3(19)$, $\ff_3(20)$, $\ff_3(29)$.
Since the upper motives of $X_2$ and $X_4$ are isomorphic, we get
$$\M(X_4)=M_{j_1,j_2}\oplus M_{j_1,j_2}(9)\oplus\bigoplus_{i\in I_2} R_{j_1,j_2}(i)$$
for some multiset of indexes $I_2$.
\end{proof}

Note that
\begin{align*}
P(\E_6/P_2,t)&=\tfrac{(t^4+1)(t^{12}-1)(t^6+t^3+1)}{t-1} \quad \text{and}\\
P(\E_6/P_4,t)&=\tfrac{(t^5-1)(t^3+1)(t^8-1)(t^6+t^3+1)(t^{12}-1)}{(t-1)(t^2-1)^2}.
\end{align*}
So, to finish the proof Tables \ref{e6.dec}--\ref{e6.multi}, it suffices to compute the Poincar\'e
polynomials of $M_{2,1}$ and $M_{1,1}$, to find motivic decompositions
for $J_3(G)=(1,0)$, and to exclude the case $J_3(G)=(2,0)$.

\begin{lem}\label{lll}
$P(M_{2,1},t)=\frac{(t^4+1)(t^{12}-1)(t^6+t^3+1)}{t^2-1}$.
\end{lem}
\begin{proof}
If $2\in I_1$ (in the notation of Lemma~\ref{l4}), then by Theorem~\ref{t1}, $3\in I_1$, since for any $\alpha\in\Ch^2(\overline X_2)$
one has $\alpha\cdot h\ne 0$. And if $3\in I_1$, then $4\in I_1$,
since for any $\beta\in\Ch^3(\overline X_2)$ one has $\beta\cdot h\ne 0$.

Thus, if $I_1$ is non-empty, then it contains an index $\ge 4$.
But the Poincar\'e polynomial of $R_{2,1}$ equals $(1+t^4+t^8)(t^9-1)/(t-1)$,
in particular, has dimension $16$.

But by Lemma~\ref{l4} we have $$P(X_2,t)=(t+1)P(M_{2,1},t)+t^mP(R_{2,1},t)+Q(t)$$
where $m\ge 4$, $\deg P(M_{2,1},t)=20$ by Lemma~\ref{l3} and the polynomial $Q(t)\in\zz[t]$ has non-negative
coefficients. Comparing the terms gives a contradiction.
\end{proof}

\begin{lem}\label{cm}
Assume $J_3(G)=(1,1)$. Then there exists a direct summand of the motive of $X_2$
starting in codimension $4$.
\end{lem}
\begin{proof}
Let $X=X'=X_2$.
A direct computation of all parameters of Proposition~\ref{lchme} shows that
\begin{align*}
\CH^*(\E_6/P_2\times\E_6/P_2)\simeq\CH^*(\E_6/P_2)\oplus\CH^{*-1}(\E_6/P_{2,4})\oplus\CH^{*-6}(\E_6/P_{1,2,6})\\
\oplus\CH^{*-11}(\E_6/P_{2,4})\oplus\CH^{*-21}(\E_6/P_2),
\end{align*}
where $\E_6$ stands for the split group of type $\E_6$.

Let $h_i$ denote the generator of $\Ch^1(\BX_i)$. Since $J_3(G)=(1,1)$, by \cite[Proposition~4.2]{PS10} $h_1^3$
is rational. Consider the rational cycle
$\alpha=h_1^6\cdot c_9\in\Ch^{15}(\E_6/P_{1,2,6})$,
where $c_9$ stands for the $9$-th Chern class of the tangent bundle
to $\BX_{1,2,6}$. Another direct computation
using Section~\ref{secSteenrod}
and formulas~\eqref{ff1} and \eqref{ff2} shows that the realization
$\alpha_\star\colon\Ch^*(\BX)\to\Ch^*(\BX)$
maps $\Ch^i(\BX)$ to zero for $i\le 3$, and maps $h_2^4$ to $-h_2^4$ (mod $3$). In particular, by Theorem~\ref{tchme}
$\alpha$ defines a projector with generic point of codimension $4$.
\end{proof}

\begin{lem} \label{PM11}
$P(M_{1,1},t)={\scriptstyle t^{20}+t^{18}+t^{17}+t^{16}+t^{14}+t^{13}+
t^{12}+t^{11}+2t^{10}+t^9+t^8+t^7+t^6+t^4+t^3+t^2+1}$.
\end{lem}
\begin{proof}
If $2\in I_1$ or $3\in I_1$ (in the notation of Lemma~\ref{l4}), then the same argument as in the proof of
Lemma~\ref{lll} implies that $3,4,5,6,7\in I_1$.
We have:
$$P(X_2,t)=P(M_{1,1},t)(1+t)+P(R_{1,1},t)Q_{1,1}(t)$$
with $Q_{1,1}(t)=t^3+t^4+t^5+t^6+t^7+Q(t)$
and $$P(M_{1,1},t)=1+t^{10}+t^{20}+P(R_{1,0},t)S(t)$$
for some polynomials $Q$ and $S$ with non-negative coefficients. Comparing the terms we come to a contradiction.

Thus, $2$ and $3\not\in I_1$. By Lemma~\ref{cm} $4\in I_1$.
Therefore $5,6,7\in I_1$.

Since $2,3\not\in I_1$, these codimensions belong to the upper motive $M_{1,1}$.
Therefore $P(M_{1,1},t)=1+t^{10}+t^{20}+t^2+t^3+Q_1$ for some $Q_1\in\zz[t]$ with non-negative coefficients.
Since $P(M_{1,1})-(1+t^{10}+t^{20})$ is divisible by $1+t+t^2$, we have $Q_1=t^4+Q_2$ for some $Q_2\in\zz[t]$
with non-negative coefficients.

By symmetry of the projector, $Q_2=t^{18}+t^{17}+t^{16}+Q_3$ for some $Q_3\in\zz[t]$ with non-negative coefficients,
and this together with above polynomial identities implies that $Q_3=Q_4\cdot t^6$ for some $Q_4\in\zz[t]$,
and $\deg Q_3=15<\dim R_{1,1}+6=16$. Therefore $I_1\subset\{4,5,6,7\}$, and, thus, $I_1=\{4,5,6,7\}$.
\end{proof}

In the following statements we assume that $\car k=0$ so that me way apply Proposition \ref{prpr}. However, we will
remove this restriction in Corollary \ref{nochar0}.

\begin{lem}\label{l9}
If $J_3(G)=(1,0)$ and $\car k=0$, then $X_2$ has a zero-cycle of degree
coprime to $3$, and in particular 
$M_{1,0}\simeq\ff_3$.
\end{lem}
\begin{proof}
Assume $X_2$ has no zero-cycles of degree coprime to $3$.
Let $A$ be the Tits algebra of $G$ and $D$ the underlying division algebra. Denote by $Y$
the Severi-Brauer variety $\SB(D)$ of $D$.
Since $J_3(G)=(1,0)$, $\ind A=3$ and by \cite[Theorem~5.7(3)]{PS10}
the variety $X_2$ is not generically split, and, hence, $\ind(A_{k(X_2)})=3$.
Therefore the motive of $Y_{k(X_2)}$ is indecomposable \cite[Th.~2.2.1]{Karp:SB}.

Moreover, over $k(X_2)$ the motive $\M(X_2)$ is isomorphic to
\[
\oplus_{i=0,1,10,11,20,21}\ff_3(i)\oplus(\oplus_{j\in J}\M(Y_{k(X_2)})(j))
\]
for some multiset of indexes $J$ by \cite[Theorem~7.5]{CGM}.

Pick a generator $h$ of the Picard group of $X_2$.
The proof of Lemma~\ref{l3} shows that this is a rational cycle from the first shell.
Now all conditions of Proposition~\ref{prpr} are satisfied and the parameter $b$ in that
proposition equals $10$. This is a contradiction, because $10\ne\frac{3^n-1}{2}$ for any $n$.
\end{proof}

\begin{cor}\label{cor10}
If $\car k=0$ and $J_3(G) = (1,0)$, then
\begin{align*}
\M(X_2)&\simeq\oplus_{i=0,1,10,11,20,21}\ff_3(i)\oplus(\oplus_{j\in
J^{1,0}_2}\M(\SB(D))(j)) \quad \text{and} \\
\M(X_4)&\simeq\oplus_{i=0,1,9,10,10,11,19,20,20,21,29,30}\ff_3(i)\oplus(\oplus_{j\in J^{1,0}_4}\M(\SB(D))(j)).
\end{align*}
\end{cor}

\begin{cor}\label{cor11}
Assume that $\car k=0$. Then $J_3(G)\ne(2,0)$.
\end{cor}
\begin{proof}
Let $A$ be the Tits algebra of $G$ and $D$ the underlying division algebra. The index of $A$ equals $3^i$ for some
$i=0,\ldots,3$.

Assume $J_3(G)=(2,0)$. Then the Borel variety and $\SB(A)$ have a common upper motive.
In particular, the Poincar\'e polynomial of this motive equals $\frac{t^{3^{j_1}}-1}{t-1}$.
Hence, $\ind A=3^{j_1}=9$.

Let $K=k(\SB_3(D))$. Then by the index reduction formula $\ind D_K=3$ (see \cite{SchVB92}).
Therefore $J_3(G_K)=(1,0)$. (The second entry is zero because each entry in the $J$-invariant is non-increasing under field extensions.)

Since $J_3(G)=(2,0)$, the variety $X_2$ has no zero-cycles
of degree coprime to $3$ (see Lemma~\ref{l2}). Therefore by Lemma~\ref{l9} $(X_2)_K$ has a zero-cycle of degree
$1$ mod $3$.

On the other hand, since
by the index reduction formula
$\ind D_{k(X_2)}=3$ (see \cite{MPW96}), the variety $\SB_3(D)_{k(X_2)}$ has a rational point.
Thus, by Lemma~\ref{lx} the motives $\U(X_2)$ and $\U(\SB_3(D))$ are isomorphic.

By Lemma~\ref{l3} $\dim \U(X_2)=20$. On the other hand,
\[
\dim \U(\SB_3(D))\le\dim\SB_3(D)=\dim\mathrm{Gr}(3,9)=18<20,
\]
which is a contradiction.
\end{proof}

\section{Reduction to characteristic zero}\label{sec9}

We now prove a general mechanism for transferring results from characteristic $0$
to a field of prime characteristic.

Fix  a prime number $\ell$ and $m \ge 1$.  Construct a complete discrete valuation ring $R$ with residue field $k$ of characteristic $p$ (possibly equal to 0 or $\ell$) and fraction field $K$ of characteristic zero.  In case $\ell = p$, we enlarge $R$ if necessary to include the $\ell^m$-th roots of unity.  We have a split exact sequence:
\begin{equation} \label{residue.seq}
0\xrightarrow{}H^{d+1}(k,\Z/\ell^m\Z(d))\xrightarrow{i^K_k}H^{d+1}_\nr(K,\mu_{\ell^m}^{\otimes d})
\xrightarrow{\partial_K}H^d(k, \Z/\ell^m \Z(d-1)) \xrightarrow{} 0
\end{equation}
where $H^{d+1}_\nr$ denotes the subgroup of elements $x$ such that $nx$ is killed by the maximal unramified
extension of $K$ for some $n$ not divisible by $p$, see \cite[p.~18]{GMS} if $p \ne \ell$ and
\cite[Th.~3 and p.~235]{Kato:cdv} if $p = \ell$. The explicit formulas for $i^K_k$ shows that it
sends symbols in $H^{d+1}(k, \Z/\ell^m \Z(d))$ to symbols in $\ker \partial_K$.  When $m = 1$, \cite[16.1]{GPe:wild} gives the converse that symbols in $\ker \partial_K$ are images of symbols in $H^{d+1}(k, \Z/\ell \Z(d))$. 

\begin{lem}
In the above notation an element $\xi\in H^{d+1}(k, \Z/\ell\Z(d))$ is a symbol over some finite extension of $k$
of degree not divisible by $\ell$ if and only if there is a finite extension of
$K$ not divisible by $\ell$ over which $i^K_k(\xi)$ is a symbol.
\end{lem}

\begin{proof}
The ``if'' direction is clear, using that symbols in the image of $i^K_k$ are images of symbols.  For ``only if'', one immediately reduces to the case where the given extension $E$ of $k$ is purely inseparable.
But, since $[E:k]$ is not divisible
by $\ell$, we have $\ell\ne p$, and the mod-$\ell$ Galois cohomology groups over $k$ and $E$ are the same, so in that case $\xi$ is already a symbol over $k$.
\end{proof}

\begin{lem}  \label{Rost99}
If $\xi \in H^{d+1}(k, \Z/2\Z(d))$ is such that $\res_{L/k}(\xi)$ is a symbol for some odd-degree extension $L$ of $k$, then $\xi$ is a symbol.
\end{lem}

\begin{proof}
If $\car k \ne 2$, the claim concerns the Galois cohomology group $H^{d+1}(k, \Z/2\Z)$, and the lemma is a result of Rost \cite{Ro99}.  Otherwise, $\car k = 2$ and we take $R$ and $K$ as above with $\ell = 2$ and $m = 1$.  Combining Rost's result and the previous lemma completes the proof.
\end{proof}

Here is the promised reduction:
\begin{prop}\label{gille}
Let $G$ be a simple simply connected linear algebraic group over $k$ and let $\ell^m$ be the largest power of the prime $\ell$ dividing the order of the Rost invariant $r_G$.  Define $R$ and $K$ as above. Then:
\begin{enumerate}
\item There is a simple simply connected linear algebraic group $H$ over $K$ that has the same Dynkin type and the same Tits index as $G$.
\item For every $\xi \in H^1(k, G)$, there is a $\zeta \in H^1(K, H)$ so that:
\begin{enumerate}
\item[(a)] The mod-$\ell$ component of $r_G(\xi)$ is zero in $H^3(k, \Z/\ell^m\Z(2))$ (resp., a symbol in $H^3(k, \Z/\ell\Z(2))$ if and only if the mod-$\ell$ component of $r_H(\zeta)$ is zero in $H^3(K, \Z/\ell^m\Z(2))$ (resp., a symbol in $H^3(K, \Z/\ell\Z(2))$).  If the mod-$\ell$ component of $r_G(\xi)$ is a sum of $\le r$ symbols in $H^3(k, \Z/\ell^m \Z(2))$ with a common slot, then the mod-$\ell$ component of $r_H(\zeta)$ is a sum of $\le r$ symbols in $H^3(K, \Z/\ell^m\Z(2))$ with a common slot.
\item[(b)] For every finite extension $L/K$, the Tits indexes of the twisted forms $(_\zeta H)_L$ and $(_\xi G)_{\Lb}$ are equal, where $\Lb$ is the residue field of $L$.
\item[(c)] For $X_\Theta^\zeta$ a twisted flag variety for $_\zeta H$ and $X_\Theta^\xi$ the corresponding variety for $_\xi G$, 
\[
\deg \CH_0 X_\Theta^\zeta = \deg \CH_0 X_\Theta^\xi
\]
as subgroups of $\Z$.
\end{enumerate}
\end{enumerate}
\end{prop}
\begin{proof}
We can find a semisimple group scheme $\cG$ over $R$ of the same Dynkin type as $G$ whose special fiber is $G$
and whose generic fiber $\cG_K$ is also of the same Dynkin type as $G$. Denote it by $H$.  
One can lift $\xi$ to a class in $H^1_\et(R, \cG)$
which we also denote by $\xi$.
Let $\zeta$ be the image of $\xi$ in $H^1(K, \cG_K)$.
By \cite[Theorem~2]{Gille:inv} one has a commutative diagram
\[
\xymatrix@R=1pc{
H^1(K,H) \ar[r]^{r_H} & H^3(K,\qq/\zz(2)) & \ar[l] H^3(K, \Z/\ell^m \Z(2)) \\
H^1_\et(R,\cG)\ar[u]\ar[d] & & \\
H^1(k,G)\ar[r]^{r_G} & H^3(k,\qq/\zz(2))\ar[uu]_{i^K_k\circ h_*} & \ar[l] H^3(k, \Z/\ell^m \Z(2)) \ar[uu]_{\pm i^K_k} 
}
\]
where $h_*$ is an automorphism that restricts to $\pm 1$ on $H^3(k, \Z/\ell^m\Z)$, hence the first sentence of  (a).  For the second sentence, the explicit formulas for $i^K_k$ show that it sends a sum of $\le r$ symbols with a common slot to a sum of $\le r$ symbols with a common slot.

The Tits indexes of $(_\zeta H)_L$ and $(_\xi G)_{\Lb}$ are the same by \cite[Expos\'e 26, 7.15]{SGA3.3}, hence (b).  It follows that $\deg \CH_0 X_\Theta^\zeta \subseteq \deg \CH_0 X_\Theta^\xi$.  For equality, in view of (b) it suffices to check that for every finite extension $k'$ of $k$, there exists an extension $K'$ of $K$ with residue field $k'$ such that $[K':K] = [k':k]$, which is a routine exercise because the valuation is Henselian.
\end{proof}

Claim (2b) in the proposition is well known in some special cases, for example when $G$ is the special orthogonal group of a quadratic form as in \cite[Prop.~VI.1.9(1)]{Lam} or $G$ is $\PGL_n$  as in \cite[Th.~2.8(b)]{JacobWadsworth} (which does not require the valuation to be discrete).

We illustrate the proposition by applying it in its typical manner.  We number the simple roots of $\E_7$ as in \eqref{E78.roots}.
\begin{cor} \label{E7.7}
Let $G$ be a simple algebraic group of type $\E_7$ over a field $k$.  If the twisted flag variety $X_7$ has a zero-cycle of odd degree, then $X_7$ has a $k$-point.
\end{cor}

\begin{proof}
The claim holds when $\car k = 0$ by \cite[Cor.~3.5]{GiSem}, so assume $\car k$ is prime.  For $K$ as defined earlier in this section, Prop.~\ref{gille}(2c) gives that $X_7(K)$ has a zero-cycle of odd degree, hence $X_7(K) \ne \emptyset$ by ibid., hence $X_7(k) \ne \emptyset$ by Prop.~\ref{gille}(2b).
\end{proof}

\section{Applications to the Rost invariant}\label{s88}

\subsection{Type $\E_6$}

We now return to the setting of \S\ref{s8}.

\begin{lem}\label{iso.lem}
Let $G$ be a group of inner type $\E_6$ over a field $k$. 
If $J_3(G)=(0,0)$, then $G$ is isotropic.
\end{lem}
\begin{proof}
By \cite[Corollary~6.7]{PSZ} since $J_3(G)=(0,0)$, $G$ splits over a field extension of $k$ of degree coprime to $3$.  Therefore the Tits algebra of $G$ (of degree 27) is split, so we may speak of the Rost invariant of $G$.  Clearly, its 3-component must be zero.

If $\car k\ne 2,3$, then by \cite{Rost:CR} the variety $X_1$ has a rational point. Proposition~\ref{gille}
implies that the same holds over any field of prime characteristic. In particular, $G$ is isotropic.
\end{proof}

\begin{lem}\label{iso.prop}
Let $G$ be a group of inner type $\E_6$ and $A$ a Tits algebra of $G$.
Assume that $\ind A\le 3$. Then
$G \times k(\SB(A))$ is isotropic if and only if $X_2$ has a zero-cycle of degree not divisible by $3$.
\end{lem} 
\begin{proof}
Suppose first that $\car k = 0$, $G$ is anisotropic, $G_{k(\SB(A))}$ is isotropic, and every zero-cycle of $X_2$ has degree divisible by 3.  We know by Lemma~\ref{l9}, Corollary~\ref{cor11}, and Lemma \ref{iso.lem} that $j_2=1$.
Since $G_{k(\SB(A))}$ is isotropic and $A_{k(\SB(A))}$ is split, $X_2$ has a zero-cycle over $k(\SB(A))$ of degree 1 or 2, hence $J_3(G_{k(\SB(A))})=(0,0)$ by Lemma \ref{l2}.

On the other hand, $\ind A_{k(X_\Delta)}=1$. Therefore, the upper motives
$\U(X_\Delta)$ and $\U(\SB(A))$ are isomorphic. Their Poincar\'e polynomials equal
$$(1+t^4+t^8)(t^{3^{j_1}}-1)/(t-1)$$ and $(t^{\ind A}-1)/(t-1)$ respectively.
In particular, they are not equal for any values of $j_1$ and $\ind A$.
Contradiction, so the ``only if" direction is proved if $\car k = 0$ and $G$ is anisotropic; this is the crux case.

If $G$ is isotropic, then it is split or has semisimple anisotropic kernel of type $2\A_2$ or $\D_4$.
In the first two cases, $X_2$ has a rational point and in the third case it has a point over a quadratic extension
of $k$.  Thus we have proved ``only if" when $\car k = 0$ or $G$ is isotropic.

So consider the case where $\car k$ is a prime $p$, $G_{k(\SB(A))}$ is isotropic, and $G$ is anisotropic;
in particular, $A$ is not split, hence, by our assumptions has index $3$.
Then there is a simply connected isotropic group $G'$
(with anisotropic kernel of type $2\A_2$) and a class $\eta \in H^1(k, G')$ such that $G$ is isomorphic to
 $G'$ twisted by $\eta$.  We control the mod-3 portion $r_{G'}(\eta)_3$ of the Rost invariant of $\eta$, which belongs
 to $H^3(k, \zz/3\zz(2))$.  Clearly, $G'$ is split by $k(\SB(A))$, so our hypothesis on $G$ gives that $k(\SB(A))$
 kills $r_{G'}(\eta)_3$.  It follows that $r_{G'}(\eta)_3 = (\zeta) \cdot [A]$ for some 
$\zeta \in k^{\times} / k^{\times 3}$ by \cite{Peyre:deg3} and \cite[Prop.~5.1]{Karp:cod},
 hence by \cite{GQ} we may replace $\eta$ by a twist by the class of a cocycle with values in the center of $G'$ and so assume 
that $r_{G'}(\eta)_3$ is zero.

One can find a simply-connected group $H$ of inner type $\E_6$
over a field $K$ of characteristic $0$ lifting $G'$ and $\zeta\in H^1(K,H)$ lifting $\eta$ as in
Proposition~\ref{gille}. In particular, $r_H(\zeta)_3=0$. Denote by $A_H$ the Tits algebra of $H$.
By \cite{Rost:CR} the twisted form is isotropic over $K(\SB(A_H))$,
and, thus, by Proposition~\ref{gille}(2c) and the characteristic zero case, we have proved the  ``only if'' part.

Now suppose that there is an extension $L/k$ of degree not divisible by 3 so that $X_2(L)$ is not empty.
If $A$ has index $1$, then $J_3(G) = (0,0)$ by Lemma \ref{l2}, and so $G$ is $k$-isotropic by Lemma~\ref{iso.lem}.  If $A$ has index $3$, then $L \otimes_k k(\SB(A))$ is a field of dimension not divisible by $3$ over $k(\SB(A))$,
hence the ``if" statement follows by the index $1$ case.
\end{proof}

\begin{rem}                                  
In case $\car k \ne 2$, one can use the Rost invariant to define a class $r(G) \in H^3(k, \zz/2\zz)$ depending only
on $G$, see \cite[\S7]{GGi}.  If $L/k$ is an extension such that $X_2(L)$ is nonempty, then certainly $r(G)$ is
killed by $L$, hence $[L:k] r(G) = 0$.  It follows that  $\deg \CH_0(X_2)$ is contained in $o(r(G))\zz$,
for $o(r(G))$ the order of $r(G)$, which is $1$ or $2$.  One can show that \emph{the conditions in Proposition \ref{iso.prop} are equivalent to $\deg \CH_0(X_2) = o(r(G)) \zz$.}
\end{rem}

\begin{cor} \label{nochar0}
Lemma~\ref{l9}, Corollary~\ref{cor10}, and Corollary~\ref{cor11} hold in any characteristic.
\end{cor}
\begin{proof}
Clearly, it suffices to prove only Lemma~\ref{l9}, so assume $J_3(G)=(1,0)$.  Then $G$ is split by an extension of degree not divisible by 9 \cite[Prop.~6.6]{PSZ}, so $\ind A=3$ and
$J_3(G_{k(\SB(A))})=(0,0)$.
Therefore by Lemma~\ref{iso.lem} $G_{k(\SB(A))}$ is isotropic and by Lemma~\ref{iso.prop}
$X_2$ has a zero-cycle of degree coprime to $3$.
\end{proof}

\begin{cor} \label{E6ind}
Let $G$ be a group of inner type $\E_6$ with Tits algebra $A$.
If $G \times k(\SB(A))$ 
is isotropic, then $A$ has index dividing $3$.
\end{cor}

\begin{proof}
Since Lemma~\ref{l9} and Corollary~\ref{cor11} hold in any characteristic, we
can repeat the first two paragraphs of the proof of Lemma~\ref{iso.prop} without any restriction on the characteristic
of $k$ to see that $X_2$ has a zero-cycle of degree not divisible by 3.
\end{proof}

We summarize the relationship between the mod-3 $J$-invariant of $G$ and its Tits index and Tits algebra in Table \ref{e6.dict}.
We use here that by \cite[Prop.~5.3]{Jun11} $j_1=1$ iff $\ind A=3$.
\begin{table}[hbt]
\begin{center}
\begin{tabular}{c|ccccc}
$J_3(G)$&$(0,0)$&$(1,0)$&$(0,1)$&$(1,1)$&$(2,1)$ \\ \hline
\rb{Tits index of $G$}&\rb{split}&\begin{picture}(5,2)
    \multiput(0.5,0.5)(1,0){5}{\circle*{\darkrad}}
    \put(2.5,1.5){\circle*{\darkrad}}

    \put(0.5,0.5){\line(1,0){4}}
    \put(2.5,1.5){\line(0,-1){1}}

    \put(2.5,0.5){\circle{\lrad}}
    \put(2.5,1.5){\circle{\lrad}}
\end{picture}
&\multicolumn{3}{c}{\rb{$\cdots$ anisotropic $\cdots$}}\\
index of $A$&1&3&1&3&9 or 27 
\end{tabular}                           
\caption{Dictionary relating the mod-3 $J$-invariant of $G$, the Tits index of $G$ over a 3-closure of $k$, and the Tits algebra $A$ of $G$} \label{e6.dict}
\end{center}
\end{table}

\begin{prop}\label{kere6}
Let $G$ be a simply connected group of inner type $\E_6$ over $k$ such that $X_2$ has a zero-cycle of degree $1$.
Write $Z$ for the center of $G$.
\begin{enumerate}
\item \label{kere6.1} The Rost invariant $r_G$ is injective on the image of $H^1(k, Z) \to H^1(k, G)$.
\item \label{kere6.2} For $\xi\in H^1(k,G)$, if the mod-$3$ component of the Rost invariant
$r_G(\xi)$ is a symbol,
then $$\gcd\{[L:k]\mid L\text{ kills }\xi\}=o(r_G(\xi)).$$
\end{enumerate}
\end{prop}
\begin{proof}
Write $A$ for the Tits algebra of $G$.  If $A$ is split, then $G$ is split
and $H^1(k, Z)$ has zero image in $H^1(k, G)$, so \eqref{kere6.1} holds.  If
$A$ has index 3, then $H^1(k, Z)$ is identified with $k^\times / k^{\times 3}$
and the composition $$H^1(k, Z) \to H^1(k, G) \to H^3(k, \Q/\Z(2))$$ is $x \mapsto \pm x \cdot [A]$ by \cite{GQ}.  By twisting, it suffices to show that this map has zero kernel.  But if $x \cdot [A]$ is zero, then $x$ is a reduced norm from $A$, i.e., there is a cubic extension $L$ of $k$ in $A$ so that $x = N_{L/k}(y)$ for some $y \in L$ by Merkurjev-Suslin if $\car k \ne 3$ and by \cite[Th.~6a]{Gille:inv} if $\car k =3$.  Now $L$ splits $A$, so $G$ is $L$-split and $y$ is in the kernel of $H^1(L, Z) \to H^1(L, G)$.  As $G/Z$ is rational as a variety over $L$, the Gille-Merkurjev Norm Principle implies that $x$ is in the kernel of $H^1(k, Z) \to H^1(k, G)$, completing the proof of \eqref{kere6.1}.

As for \eqref{kere6.2}, one quickly reduces to the case where $r_G(\xi)$ is zero (because the mod-2 and 
  mod-3 components of $r_G(\xi)$ are symbols --- for 2 this is by Lemma \ref{Rost99}), $X_2$ has a rational point, and $A$ has index 3.
There is a cubic extension of $k$ splitting $A$, hence splitting $G$, hence killing $\xi$.  On the other hand, ${_\xi}G \times k(\SB(A))$ is split, so by Lemma \ref{iso.prop} the ${_\xi}G$-variety $X_2$ has a point over extensions $L_1, \ldots, L_r$ such that $\gcd\{ [L_i:k] \}$ is not divisible by 3.  Over each $L_i$, $\xi$ is in the kernel of the map $H^1(L_i, G) \to H^1(L_i, G/Z)$ by Tits's Witt-type Theorem, so is equivalent to the class of a cocycle $z$ with values in $Z$.  By \eqref{kere6.1}, $\xi$ is killed by $L_i$.  This proves \eqref{kere6.2}.
\end{proof}

\subsection{Type $\E_7$.}

For use in this subsection and the next, we recall that the simple roots of $\E_7$ and $\E_8$ are numbered like this:
\begin{equation} \label{E78.roots}
\begin{picture}(0,1.9)
\put(0,0.4){\makebox(0,0.4)[b]{$\E_7$}}
\end{picture}
\begin{picture}(7,1.9)
    \multiput(1,0.7)(1,0){6}{\circle*{\darkrad}}
    \put(3,1.45){\circle*{\darkrad}}

    \put(1,.7){\line(1,0){5}}
    \put(3,1.45){\line(0,-1){0.75}}
    
    \put(1,0){\makebox(0,0.4)[b]{$1$}}
    \put(2,0){\makebox(0,0.4)[b]{$3$}}
    \put(3,0){\makebox(0,0.4)[b]{$4$}}
    \put(4,0){\makebox(0,0.4)[b]{$5$}}
    \put(5,0){\makebox(0,0.4)[b]{$6$}}
        \put(6,0){\makebox(0,0.4)[b]{$7$}}
    \put(3.3,1.3){\makebox(0,0.4)[b]{$2$}}
\end{picture} \quad
\begin{picture}(8,1.9)
    \multiput(1,0.7)(1,0){7}{\circle*{\darkrad}}
    \put(3,1.45){\circle*{\darkrad}}

    \put(1,.7){\line(1,0){6}}
    \put(3,1.45){\line(0,-1){0.75}}
    
    \put(1,0){\makebox(0,0.4)[b]{$1$}}
    \put(2,0){\makebox(0,0.4)[b]{$3$}}
    \put(3,0){\makebox(0,0.4)[b]{$4$}}
    \put(4,0){\makebox(0,0.4)[b]{$5$}}
    \put(5,0){\makebox(0,0.4)[b]{$6$}}
        \put(6,0){\makebox(0,0.4)[b]{$7$}}
                \put(7,0){\makebox(0,0.4)[b]{$8$}}
    \put(3.3,1.3){\makebox(0,0.4)[b]{$2$}}
\end{picture}
\ \begin{picture}(1,1.9)
\put(0,0.4){\makebox(0,0.4)[b]{$\E_8$}}
\end{picture}
\end{equation}
A group $G$ of type $\E_7$ has (essentially) one Tits algebra, as explained in \cite[6.5.1]{Ti71}.  It is a central simple algebra of exponent dividing 2 and index dividing 8.

\begin{prop}
Let $G$ be an anisotropic group of type $\E_7$ with Tits algebra $H$. If $G_{k(\SB(H))}$ is split, then $\ind H=2$.
\end{prop}
\begin{proof}
Let $J_2(G)=(j_1,j_2,j_3,j_4)$, $j_i=0,1$, be the $J$-invariant of adjoint
$\E_7$ (see \cite[Section~4.13]{PSZ}).
Since $G$ is anisotropic and $G_{k(\SB(H))}$ is split, $j_1=1$ by
\cite[Proposition~4.2]{PS10}.

Moreover, the upper motives of the variety of Borel subgroups $X_\Delta$ and of $\SB(H)$
are isomorphic. Their Poincar\'e polynomials equal 
\[
(1+t)(1+t^3)^{j_2}(1+t^5)^{j_3}(1+t^9)^{j_4} \quad \text{and} \quad \dfrac{t^{\ind H}-1}{t-1}.
\]
Since they are equal, we have $j_2=j_3=j_4=0$ and $\ind H=2$.
\end{proof}

The following lemma provides a crucial computation for the proof of Theorems \ref{prope6} and  \ref{pro816} below, which settle Rost's question described in the introduction.  The proof involves a computer calculation, which we did via two independent methods: the one described in section \ref{secSteenrod} and the one in \cite{DuZ07}.  Alternatively, the paper \cite{KI} computes the Steenrod operations on $\Ch(\BX_\Delta)$, and presumably the computer calculation here could be replaced by an argument connecting their computation with $X_7$.

\begin{lem}\label{le7}
Assume that the variety $X_7$ does not have a zero-cycle of odd degree,
the Tits algebras of $G$ are split, and $\car k\ne 2$.
Then $\ff_2(9)$ is a direct summand of the motive $\U(X_7)$ over $k(X_7)$.
\end{lem}
\begin{proof}
Let $h\in\Ch^1(\BX_7)$, $e_5\in\Ch^5(\BX_7)$, and $e_9\in\Ch^9(\BX_7)$
denote some Schubert cycles. Then independently of the choice of these
cycles, the elements $h^9$, $e_5h^4$, and $e_9$ form an $\ff_2$-basis
of $\Ch^9(\BX_7)$. Note also that the cycle $h$ is rational, since the
Tits algebras of $G$ are split.

We claim that the cycles $e_5h^4$, $e_9$, and $e_5h^4+e_9$ are not rational.

Indeed, a direct computation of Steenrod operations modulo $2$
shows that
$$e_5h^5\cdot S^8(e_5h^4)=e_9h\cdot S^8(e_9)=(e_5h^5+e_9h)S^8(e_5h^4+e_9)=\pt,$$
where $\pt$ denotes the class of a rational point on $\BX_7$.
Since by our assumptions $X_7$ has no zero-cycles of odd degree, the only rational cycle
in $\Ch^9(\BX_7)$ is $h^9$.

But the cycle $h^9$ does not lie in the first shell.
Indeed, an explicit computation of the decomposition
of \cite[Theorem~7.5]{CGM} for $\M(X_7)$ shows that over $k(X_7)$ its motive
contains exactly the following Tate motives:
$\ff_2$, $\ff_2(1)$, $\ff_2(9)$, $\ff_2(10)$, $\ff_2(17)$, $\ff_2(18)$,
$\ff_2(26)$, and $\ff_2(27)$, and that the cycle from the dual codimension
that corresponds to the Tate motive $\ff_2(9)$ equals
$Z_{[1,3,4,2,5,4,3,1,7,6,5,4,2,3,4,5,6,7]}$
in the notation of Section~\ref{secSteenrod}.
A direct computation using Poincar\'e duality shows that this cycle is not dual to $h^9$.

Since generic points of direct summands of $X_7$ are rational, no shift of
$\U(X_7)$ of $X_7$ starts in codimension $9$. Therefore the Tate motive $\ff_2(9)$, which
belongs to the first shell, must be a summand of $\U(X_7)_{k(X_7)}$.
\end{proof}

\begin{lem}\label{l10}
Assume that $G$ is anisotropic, the variety $X_7$ has no zero-cycles of odd degree, the
Tits algebras of $G$ are split, and $\car k\ne 2$. Then the height of $X_1$ equals $3$
and $G_{k(X_1)}$ has semisimple anisotropic kernel of type $\D_6$.
\end{lem}
\begin{proof}
The $(2, \{ 1 \})$-indexes of $G$ are $\{ 1 \}$, $\Psi := \{ 1, 6, 7 \}$, and $\Delta$, so by Example \ref{tower}
the nonempty shells on $X_1$ are the first shell, which is contained in $\SH_{\leqslant \Psi}$, which is contained
in the last shell.

By \cite[Th.~5.7(6)]{PS10} the varieties $X_1$ and $X_7$ are not generically split.
Therefore by the Tits classification \cite{Ti66} the height of $X_1$ is $2$ or $3$.
Assume that it is two.
Then the upper motives $\U(X_7)$ and $\U(X_1)$ are isomorphic.

By Lemma~\ref{le7} the motive $\U(X_7)$ has the property
that $\ff_2(9)$ is its direct summand over $k(X_7)$.
On the other hand, a direct computation using \cite[Th.~7.5]{CGM} shows that
$\ff_2(9)$ is not a direct summand of the motive of $X_1$ over $k(X_7)$.
Contradiction.
\end{proof}                            

\begin{thm}\label{prope6}
Let $G_0$ be a split, simply connected group of type $\E_7$,
$\xi\in H^1(k,G_0)$, and $G={}_\xi(G_0)$.
The following conditions are equivalent
\begin{enumerate}
\item \label{prope6.rost} $6r_{G_0}(\xi)=0$ and the mod-$2$ component of
$r_{G_0}(\xi)$ is a symbol;
\item \label{prope6.var} the $G$-variety $X_7$ has a rational point;
\item \label{prope6.cocycle} the element $\xi$ lifts to $H^1(k,\E_6)$, where $\E_6$ stands
for the split simply connected group of type $\E_6$.
\end{enumerate}
\end{thm}

This result settles the question raised in the last sentence of \cite{Sp06}, which asked whether there exists a criterion on $r_{G_0}(\xi)$ for whether \eqref{prope6.cocycle} holds.

The proof uses the following notion: suppose $\alpha \!: H \to G_0$ is a homomorphism of absolutely simple simply connected groups.  Then the composition $H^1(-, H) \xrightarrow{\alpha} H^1(-, G_0) \xrightarrow{r_{G_0}} H^3(-, \Q/\Z(2))$ equals $n_\alpha r_H$ for some nonnegative integer $n_\alpha$ called the \emph{Dynkin index} of $\alpha$, see \cite[pp.~122, 123]{GMS} for its basic properties.  It can be calculated from root system data as follows.  Let $T$ be a maximal torus of $G_0$ and $U$ be a maximal torus in $H$ such that $\alpha(U) \subseteq T$.  There is a unique Weyl-invariant quadratic form $q$ on the coroot lattice $\Hom(\Gm, T)$ that takes the value 1 on short coroots, and the Dynkin index $n_\alpha$ is the value of $q$ on the image under $\alpha$ of a short coroot in $\Hom(\Gm, U)$.

\begin{proof}[Proof of Theorem \ref{prope6}]
Assume \eqref{prope6.rost}, and that \eqref{prope6.var} fails; we seek a contradiction.
By Proposition~\ref{gille} we may assume that $\car k=0$.
A $X_7$ has no rational point, it has no zero-cycle of odd degree as in Cor.~\ref{E7.7}.

By Lemma~\ref{l10}, the anisotropic kernel
of $G_{k(X_1)}$ has type $\D_6$ and, thus, equals $\mathrm{Spin}(q)$
for a $12$-dimensional quadratic form $q$ with trivial discriminant
and trivial Clifford invariant.  Because the inclusion $\D_6 \subset \E_7$ has Dynkin index 1,
the Arason invariant
of $q$ is also a symbol. This gives a contradiction with \cite[Lemma~12.5]{Ga09}, hence \eqref{prope6.rost} $\Rightarrow$ \eqref{prope6.var}.

\eqref{prope6.cocycle} obviously implies \eqref{prope6.rost}.
Assume \eqref{prope6.var}. By Tits's Witt-type theorem, $\xi$ is equivalent to
the class of a cocycle taking values in the parabolic subgroup $P_7$. Let $L$ be the Levi
part of $P_7$. By \cite[Exp. XXVI, Cor.~2.3]{SGA3.3} $H^1(k,P_7)=H^1(k,L)$.
Then $\xi\in H^1(k,L)$ comes from $H^1(k,\E_6)$
by the exact sequence $$1\to\E_6\to L\to\mathbb{G}_m\to 1$$
and by Hilbert 90.
\end{proof}

The split simply connected group $\E_6$ contains a split group $\Gtwo$ of that type, and the inclusion $\Gtwo \subset \E_6$ has Dynkin index 1.

\begin{cor} \label{E7.rinv}
The map $H^1(k, \Gtwo) \to H^1(k, G_0)$ identifies $H^1(k, \Gtwo)$ with the subset 
\[
\{ \xi \in H^1(k, G_0) \mid \text{$r_{G_0}(\xi)$ is a symbol in $H^3(k, \Z/2\Z(2))$} \}.
\]
For each such $\xi$, the kernel of the Rost invariant $H^1(k, {_\xi(G_0)}) \to H^3(k, \Z/12\Z(2))$ is zero.
\end{cor}

In the statement, $G_0$ is split simply connected of type $\E_7$, as in Theorem \ref{prope6}.  But note that the statement holds verbatim if $G_0$ is instead taken to be $\E_6$.

\begin{proof}
The Rost invariant $r_{\Gtwo}$ identifies $H^1(k, \Gtwo)$ with the set of symbols in $H^3(k, \Z/2\Z(2))$, cf.~\cite[p.~44]{GMS}.  Hence, as the total inclusion $\Gtwo \subset \E_6 \subset G_0$ has Dynkin index 1, the image of $H^1(k, \Gtwo)$ is contained in the displayed set.  Conversely, if $\xi, \xi'$ are in the displayed set and $r_{G_0}(\xi) = r_{G_0}(\xi')$, then $\xi, \xi'$ come from $H^1(k, \E_6)$ by Theorem \ref{prope6}, hence $\xi = \xi'$ by the analogous property for $\E_6$.  The second claim follows from the first by twisting.
\end{proof}

\begin{rem}
This corollary includes as a special case that $\ker r_{G_0} = 0$.
So it gives a fourth proof of this statement, with the first three being Garibaldi (2001), Chernousov (2003), and Petrov-Semenov \cite{PS10}.
\end{rem}

It is conjectured that
the Rost invariant $H^1(k, {_\xi(G_0)}) \to H^3(k, \Z/12\Z(2))$ has zero kernel whenever $\xi$ satisfies Theorem \ref{prope6}\eqref{prope6.rost}, see \cite[11.11]{G:lens}.

\begin{cor} \label{E7.HPC}
If $\car k = 0$, $k(\sqrt{-1})$ has cohomological dimension $\le 2$, and $r_{G_0}(\xi)$ is a symbol in $H^3(k, \Z/2\Z)$, then the natural map 
\[
H^1(k, {_\xi(G_0)}) \rightarrow \prod_{\text{orderings $v$ of $k$}} H^1(k_v, {_\xi(G_0)})
\]
has zero kernel.
\end{cor}

That is, the ``Hasse Principle Conjecture II" holds for the group $_\xi(G_0)$.  This is new.  The analogous statement in prime characteristic is Serre's ``Conjecture II", which is known for these groups by, e.g., \cite{Gille:sc}.

\begin{proof}[Proof of Cor.~\ref{E7.HPC}]
The hypothesis on $k$ gives that $$H^3(k, \Q/\Z(2)) = H^3(k, \Z/2\Z),$$ and the claim is obvious from Corollary \ref{E7.rinv} and the injectivity of the map $H^3(k, \Z/2\Z) \to \prod H^3(k_v, \Z/2\Z)$.
\end{proof}

We can also prove a new case of the local-global principle studied in \cite{ParPreeti}.  A \emph{global field} $k$ is a number field or a finite extension $\ff_p(t)$. We write $k_v$ for the completion of $k$ with respect to a valuation $v$.

\begin{cor} \label{E7.PP}
Let $C$ be a proper, smooth, and geometrically integral curve over a global field $k$.  If 
\begin{enumerate}
\item \label{E7.PP.k} $G$ is the base change to $k(C)$ of a simply connected group of type $\E_7$ with trivial Tits algebras over $k$; or 
\item \label{E7.PP.gen} $G = {_\xi(G_0)}$ for some $\xi \in H^1(k(C), G_0)$ such that $r_{G_0}(\xi)$ is a symbol in $H^3(k(C), \Z/2\Z(2))$,
\end{enumerate}
then the natural map
\[
H^1(k(C), G) \to \prod_{\text{valuation $v$ of $k$}} H^1(k_v(C), G)
\]
has zero kernel.
\end{cor}

\begin{proof}
Suppose we are in case \eqref{E7.PP.gen} and $x \in H^1(k(C), G)$ has zero image in $H^1(k_v(C), G)$ for all $v$.
Then $r_G(x)$ has zero image under $$H^3(k(C), \Z/12\Z(2)) \to \prod_v H^3(k_v(C), \Z/12\Z(2)),$$ hence $r_G(x)$ equals zero by \cite[Th.~0.8(2)]{Kato:Hasse}, and 
Corollary~\ref{E7.rinv} gives the claim.

In case \eqref{E7.PP.k}, let $\xi \in H^1(k, G_0)$ be such that $G = {_{\xi} G_0} \times k(C)$.  Then $r_{G_0}(\xi)$ belongs to $H^3(k, \Q/\Z(2)) = H^3(k, \Z/2\Z(2))$, so it is necessarily a symbol, i.e., \eqref{E7.PP.k} is a special case of \eqref{E7.PP.gen}.
\end{proof}

With this same kind of proof, combined with the arguments from \cite[\S4.3]{HHK:CMH}, one can also prove a local-global principle for $_\xi(G_0)$ as in Corollaries \ref{E7.HPC} and \ref{E7.PP}, but with the field replaced by the function field of a curve over a complete discretely valued field.

\begin{lem}\label{l11}
Let $q$ be a regular $12$-dimensional quadratic form with trivial discriminant
over a field $k$ with $\car k\ne 2$ such that the respective special orthogonal
group has  $J$-invariant $(0,1,0)$. Then $q$ is isotropic.
\end{lem}
\begin{proof}
Assume that $q$ is anisotropic.

Let $G$ be the orthogonal group corresponding to $q$.
By \cite[Prop.~4.2]{PS10} the Clifford invariant of $q$ is trivial.
Therefore by the classification of $12$-dimensional quadratic forms $q$
has splitting pattern $(2,4)$. Let $Q=X_1$ be the projective quadric given
by $q=0$ and $h\in\Ch^1(\BX_1)$ the unique Schubert cycle.

By Example \ref{quadric.eg}, there are exactly two (non-empty) shells on $Q$, namely, $\SH_{\leqslant\{1\}}$ (the first
shell) and $\SH_{\leqslant\{3\}}$.
The powers $h^i\in\Ch^i(\BX_1)$ are rational and lie in the first shell if
$i=0,1$ and in $\SH_{\leqslant\{3\}}\setminus\SH_{\leqslant\{1\}}$ if $i=2,3,4,5$.

Since $J_2(G)=(0,1,0)$,
the Poincar\'e polynomial of the upper motive $\U(X_\Delta)$ of the Borel variety
equals $t^3+1$. Moreover, since $q$ has height two, $\U(X_\Delta)_{k(Q)}$ is indecomposable.

We have the following motivic decomposition over $k(Q)$:
$$\M(Q_{k(Q)})\simeq \oplus_{i=0,1,9,10}\ff_2(i)\bigoplus\oplus_{i=2}^5\U(X_\Delta)_{k(Q)}(i).$$

So, all conditions of Proposition~\ref{prpr} are satisfied and the parameter $b$ of that
Proposition equals $9$. This is a contradiction, since $9\ne 2^n-1$ for any $n$.
(In the proof of Proposition~\ref{prpr} in case $X$ is a projective quadric, one can use
\cite[Theorem~2.1]{Vi10} instead of Corollary~\ref{cor1}. Then the restriction $\car k=0$
is substituted by the restriction $\car k\ne 2$.)
\end{proof}

\begin{prop}\label{je7}
Let $G$ be an adjoint group of type $\E_7$ with $J_2(G)=(0,1,0,0)$ over a field $k$ with
$\car k\ne 2$. Then $X_7$ has a rational point.
\end{prop}
\begin{proof}
Since $j_1=0$, the Tits algebras of $G$ are split.

If $G$ is isotropic with anisotropic kernel of type $\D_6$, then we get
a contradiction with Lemma~\ref{l11}, so 
assume that $G$ is anisotropic.  Then $X_7$ has no zero-cycles of odd degree by Corollary \ref{E7.7}, so by Lemma~\ref{l10} the height of $X_1$ equals $3$ and
the semisimple anisotropic kernel $G'$ of $G_{k(X_1)}$ has type $\D_6$. 

Since the $J$-invariant
is non-increasing under field extensions and since $G_{k(X_1)}$ is not split,
the $J$-invariant of $G_{k(X_1)}$ also equals $(0,1,0,0)$.
Therefore by \cite[Cor.~5.19]{PSZ} we have $J_2(G')=(0,1,0)$, and again we get a contradiction with Lemma~\ref{l11}.
\end{proof}

\begin{cor}\label{co815}
Let $C$ be a smooth projective irreducible curve over $\qq_p$, $G_0$
be a split simply-connected group of type $\E_7$ over
$\qq_p(C)$ and $\xi\in H^1(\qq_p(C),G_0)$. Then: 
$6r_{G_0}(\xi)=0$ iff $_\xi(G_0)$ is isotropic.
\end{cor}
\begin{proof}
Note first that the order of the Rost invariant $r_{G_0}$ is $12$ and that ``if" is easy.

Assume $6r_{G_0}(\xi) = 0$.  Then the mod-$4$ component of the Rost invariant of $\xi$ lies in $H^3(k(C),\zz/2)$ and 
so is a symbol by \cite[Th.~3.9]{PaSu98} or \cite{Leep}.
Theorem \ref{prope6} gives that $_\xi(G_0)$ is isotropic.
\end{proof}

We remark that, using the theory of Bruhat-Tits buildings J.~Tits shows in
\cite[Proposition~2(B)]{Ti90} that there is an anisotropic group of type
$\E_7$ with trivial Tits algebras over $\qq_p(t)$. That is, there exists
$\xi\in H^1(\qq_p(t),G_0)$ such that $_\xi(G_0)$ is anisotropic.

\begin{thm}\label{pro816}
Let $G_0$ be a split, simply connected group of type $\E_7$
over a field $k$, $\xi\in H^1(k,G_0)$, and $G={}_\xi(G_0)$. 
The following conditions are equivalent
\begin{enumerate}
\item \label{pro816.rost}
there is an odd-degree
extension $L/k$ so that $r_{G_0}(\xi_L)$ is a sum of two symbols in $H^3(L,\zz/2\zz(2))$ with a common slot;
\item \label{pro816.var} the $G$-variety $X_1$ has a zero-cycle of odd degree;
\item \label{pro816.3} $G$ becomes isotropic over an odd-degree extension of $k$.
\end{enumerate}
\end{thm}

\begin{proof}
The implication \eqref{pro816.var} $\Rightarrow$ \eqref{pro816.3} is trivial and \eqref{pro816.3} $\Rightarrow$ \eqref{pro816.rost}
is \cite[p.~70, Prop.~A.1]{G:lens}.  So assume \eqref{pro816.rost}; we prove \eqref{pro816.var}, 
We may replace $k$ with $L$ and so assume that $r_{G_0}(\xi)$ is a sum of two symbols in $H^3(k, \zz/2\zz(2))$.  We may assume that $\car k = 0$ by Proposition \ref{gille}.

The $J$-invariant $J_2(G)$ is $(0,0,0)$, $(1,0,0)$, $(1,1,0)$,
or $(1,1,1)$ because $G$ is simply connected. In the first two cases, $X_7$
has a rational point (by Cor.~\ref{E7.7} and Prop.~\ref{je7} respectively),
hence \eqref{pro816.var}. So we can assume that $J_2(G)=(1,1,j_3)$
for some $j_3$.  

By hypothesis, there is a regular quadratic form $q$ over $k$ of dimension $12$
whose Arason invariant equals $r_{G_0}(\xi)$.  We assume that $q$ is anisotropic, for otherwise $r_{G_0}(\xi)$ is a symbol and $X_7$ has a rational point by Theorem \ref{prope6}, hence \eqref{pro816.var}.  
We denote the corresponding
projective quadric by $Q$. 
Over $k(X_7)$ the Rost invariant $r_{G_0}(\xi)$ is a symbol, hence 
the form $q$ is isotropic over $k(X_7)$. Conversely,
the Rost invariant of $\xi$ over $k(Q)$ is a symbol, so 
by Theorem \ref{prope6} $X_7$ has a $k(Q)$-point.  Therefore the upper motives $\U(X_7)$ and $\U(Q)$ are
isomorphic. Moreover, $\M(Q)\simeq\U(Q)\oplus\U(Q)(1)$.
Therefore, since $X_7$ has height $2$, we have 
\[
\M(X_7)\simeq\U(Q)\oplus\U(Q)(1)\oplus\U(Q)(17)\oplus\U(Q)(18)\oplus
\oplus_{i\in I}\U(X_\Delta)(i),
\]
where $I$ is some multiset of indexes. The Poincar\'e polynomial of $\U(X_\Delta)$
equals $$(t^3+1)(t^5+1)(t^9+1)^{j_3},$$
and $P(X_7,t)-(1+t+t^{17}+t^{18})P(\U(Q),t)$ is divisible by $P(\U(X_\Delta),t)$.
An easy computation shows then that $j_3=0$.

Consider now the variety $X_1$ over $K:=k(X_1)$.
A direct computation using \cite[Theorem~7.5]{CGM} gives the following
decomposition over $K$:
\[
\M(X_1)\simeq\ff_2\oplus \M(X'_3)(1)\oplus \M(X'_6)(8)\oplus \M(X'_3)(17)\oplus\ff_3(33),
\]
where $X'_3$ and $X'_6$ are $\Spin(q)$-homogeneous varieties of types $3$ and $6$
(here the enumeration of simple roots comes from the embedding $\D_6<\E_7$, i.e.,
$X'_3$ is a connected component of the maximal orthogonal Grassmannian and $X'_6$
is the variety of isotropic planes).

The variety $X'_3$ is generically split. Therefore $\M(X'_3)$ is
a direct sum over $k$ of Tate twists of the motive $\U(X_\Delta)$.
The variety $X'_6$ has height $2$ and is a direct sum over $k$ of Tate shifts of
the motives $\U(X_\Delta)$ and $\U(Q)$.

But $J_2(G_{k(X_1)})=(1,1,0)$ by Lemma~\ref{l10} and Proposition~\ref{je7}.
Therefore the motives $\U(X_\Delta)_{k(X_1)}$ and $\U(Q)_{k(X_1)}$ are indecomposable.            
If $X_1$ has no zero-cycles of odd degree, then we can apply Proposition~\ref{prpr}, which gives  a contradiction,
since $33\ne 2^n-1$ for any $n$.  Therefore, \eqref{pro816.rost} $\Rightarrow$ \eqref{pro816.var}.
\end{proof}

\begin{rem}
For $G = {_\xi}(G_0)$ of type $\E_7$, it is known that (1) for any extension $L$ of $k$ such that $X_7(L) \ne \emptyset$, we have $\res_{L/k}(4r_{G_0}(\xi)) = 0$ iff $X_1 \times L$ has a zero-cycle of degree not divisible by 3 (by \cite{Rost:CR} and Prop.~\ref{gille}) and (2) there exists a separable extension $K$ of $k$ of dimension 1 or 2 such that $X_7(K) \ne \emptyset$ (by \cite[12.13]{G:lens} and Prop.~\ref{gille}).  As $4r_{G_0}(\xi) \in H^3(k, \Z/3\Z(2))$, combining these observations gives: \emph{$X_1$ has a zero-cycle of degree 1 iff $X_1$ has a zero-cycle of odd degree and $4r_{G_0}(\xi) = 0$.}
\end{rem}

Table \ref{e7.dict} summarizes what we have proved about the relationships between the Rost invariant,
and the Tits index for groups of type $\E_7$ at the prime $2$; it also gives a description of $J$-invariant for simply connected groups of type $\E_7$ for fields of characteristic $0$.  The equivalence
for the $J$-invariant $(1,1,0)$ follows from the proof of Theorem~\ref{pro816}.

\begin{table}[hbt]
\[
\begin{array}{c|cccc}
J_2(G)&(0,0,0)&(1,0,0)&(1,1,0)&(1,1,1) \\ \hline
\rb{\text{Tits index of $G$}}&\rb{\text{split}}&
\begin{picture}(7,1.9)
    \multiput(1,0.7)(1,0){6}{\circle*{\darkrad}}
    \put(1,0.7){\circle{\lrad}}
    \put(5,0.7){\circle{\lrad}}
      \put(6,0.7){\circle{\lrad}}
     \put(3,1.45){\circle*{\darkrad}}

    \put(1,.7){\line(1,0){5}}
    \put(3,1.45){\line(0,-1){0.75}}
\end{picture}&
\begin{picture}(7,1.9)
    \multiput(1,0.7)(1,0){6}{\circle*{\darkrad}}
    \put(1,0.7){\circle{\lrad}}
     \put(3,1.45){\circle*{\darkrad}}

    \put(1,.7){\line(1,0){5}}
    \put(3,1.45){\line(0,-1){0.75}}
\end{picture}
&\rb{\text{anisotropic}}\\
r_{G_0}(\xi)&0&\parbox{1.3in}{nonzero symbol in $H^3(K, \Z/2\Z(2))$}
&\parbox{1.5in}{sum of two symbols with a common slot in $H^3(K, \Z/2\Z(2))$}&\text{otherwise}
\end{array}
\]
\caption{Dictionary relating the mod-$2$ $J$-invariant of $G$, the Tits index
of $G$ over a $2$-closure $K$ of $k$, and the Rost invariant $r_{G_0}(\xi_K)$,
for $G_0$ split simply connected of type $\E_7$.} \label{e7.dict}
\end{table}

\begin{rem}
For completeness' sake, we mention the analogous results for a group $G$ of type $\E_7$ at the prime $3$.  (The case of primes $> 3$ being trivial.)
There is an extension $L$ of $k$ of degree not divisible by $3$ over which $G$ has trivial Tits algebras and
the homogeneous variety $X_7$ has a rational point \cite[13.1]{G:lens}.
It follows that the mod-$3$ component of $r(G_L)$ is a symbol.  The mod-$3$ component of $r(G_L)$ is zero iff
$X_{1,6,7}$ has an $L$-point.
\end{rem}

\subsection{Type $\E_8$.}  

Recall the following known result:

\begin{prop}\label{kere8}
Let $G_0$ be a split group of type $\E_8$ over a field $k$, $\xi\in H^1(k,G_0)$,
$G={}_\xi(G_0)$ and $p$ an odd prime.
If the mod-$p$ component of $r_{G_0}(\xi)$ is zero,
then $G$ is split over a field extension of degree not divisible by $p$.
\end{prop}

The proposition is trivial for $p \ge 11$ and $p = 7$ amounts to noting that $7$ does not divide $120$, see \cite[p.~1135]{Ti:deg}.  The cases $p = 3, 5$ are more substantial and are the main results of two papers of Chernousov, see \cite{Ch:mod5} or \cite[Prop.~15.5]{G:lens} for the mod-5 case and \cite{Ch:mod3} for the mod-3 case.    We give a short proof of the $p = 3$ case using the methods of this paper.

\begin{proof}[Proof of Prop.~\ref{kere8} for $p = 3$]
By Proposition~\ref{gille} we can assume that $\car k=0$. Replacing $k$ by an extension 
of degree coprime to $3$, we can assume that the Rost invariant $r_{G_0}(\xi)$ is zero.
 
Consider the variety $X$ of parabolic subgroups of $G$ of type $7$. By the
classification of Tits indexes, $G$ has a parabolic of type $8$ over $k(X)$,
hence the semisimple anisotropic kernel of $G_{k(X)}$ is contained in
a simply connected subgroup of type $\E_6$.  But the Rost invariant of the split $\E_6$ has zero kernel, so it follows that $X$ is generically split.

Therefore by \cite[Th.~5.7]{PS10} $J_3(G)=(0,0)$, hence by
\cite[Prop.~3.9(3)]{PS10} $G$ splits over a field extension of degree coprime to $3$.
\end{proof}

The conclusion of Prop.~\ref{kere8} is false in general for the omitted prime $p = 2$, e.g., in case $k = \mathbb{R}$. For $p = 2$, one needs to inspect also the degree 5 invariant constructed in \cite{Sem09}.

\begin{lem}\label{ee8}
Let $G$ be a group of type $\E_8$ over a field $k$ with $\car k=0$.                 
If $J_3(G)=(1,0)$, then $X_8$ is isotropic over a field extension of degree coprime to $3$.
\end{lem}
\begin{proof}
Assume that $X_8$ has no zero-cycles of degree coprime to $3$.
By \cite[Theorem~5.7(8)]{PS10} $X_8$ is not generically split.
Therefore, since $J_3(G)=(1,0)$, the motive $\U(X_\Delta)_{k(X_8)}$ is indecomposable.

We have the following motivic decomposition over $k(X_8)$:
$$\M(X_8)\simeq\oplus_{i=0,1,28,29,56,57}\ff_3(i)\oplus \left( \oplus_{j\in J}\U(X_\Delta)(j) \right)$$
for some multiset of indexes $J$.

Moreover, the Picard group of $X_8$ is rational, since the Tits algebras of $G$ are split. It follows that
the (unique) generator of the Picard group lies in the first shell.
This leads to a contradiction with Proposition~\ref{prpr}, since $28\ne (3^n-1)/2$ for any $n$.
\end{proof}

With Lemma~\ref{ee8} in hand, we can significantly strengthen Prop.~\ref{kere8}
by giving criteria for $r_{G_0}(\xi)$ to be a symbol over an extension of
degree not divisible by some odd prime $p$. For $p \ge 5$, this happens for
every $\xi$ (see \cite[14.7, 14.13]{G:lens} for the case $p=5$).
For $p = 3$, we have:

\begin{thm} \label{E8symb}
Let $G_0$ be a split group of type $\E_8$ over a field $k$, $\xi\in H^1(k,G_0)$,
and $G={}_\xi(G_0)$.
The following conditions are equivalent:
\begin{enumerate}
\item \label{E8symb.1} $r_{G_0}(\xi)$ is a symbol over a field extension of degree coprime to $3$;
\item \label{E8symb.2} The $G$-homogeneous variety $X_{7,8}$ is isotropic over a field extension of degree coprime to $3$;
\item \label{E8symb.3} $G$ is isotropic over a field extension of degree coprime to $3$.
\end{enumerate}
\end{thm}

\begin{proof}
We assume  \eqref{E8symb.1} and prove \eqref{E8symb.2}.
Without loss of generality we can assume that the even and the mod-$5$ components of
the Rost invariant of $\xi$ are zero.

By Proposition~\ref{gille} we can assume that $\car k=0$.
Assume $r_{G_0}(\xi)$ is a symbol over a field extension of degree coprime to $3$; by Prop.~\ref{kere8}
we can assume that it is not zero. Consider its generic splitting variety $D$. The upper motive of $D$ is
a generalized Rost motive $R$ with Poincar\'e polynomial $1+t^4+t^8$
(see e.g. \cite{NSZ}).

Let $X_\Delta$ denote the variety of Borel subgroups of $G$.
Then it is obvious that $R$ splits over $k(X_\Delta)$. On the other hand,
the kernel of the Rost invariant is trivial modulo $3$ by Prop.~\ref{kere8}.
Therefore the upper motives of $D$ and $X_\Delta$ are isomorphic.
Thus, $J_3(G)=(1,0)$. By Lemma~\ref{ee8} $X_8$ is isotropic over a field
extension $L$ of degree coprime to $3$. But then $X_7$ is also isotropic over
an extension of $L$ of degree dividing $2$.

Finally, \eqref{E8symb.2} obviously implies \eqref{E8symb.3}, and \eqref{E8symb.3}
implies \eqref{E8symb.2} by the classification of possible Tits indexes in
\cite{Ti66}.  Property \eqref{E8symb.2} easily implies \eqref{E8symb.1}.
\end{proof}

\begin{rem}
If one attempts to sharpen the theorem by deleting the text ``over a field extension of degree coprime to 3'',
the implication $(2)\Rightarrow(1)$ still holds but $(1)\Rightarrow(2)$ fails.
Indeed, for $\xi \in H^1(\mathbb{R}, G_0)$ such that $_\xi(G_0)$ is the compact $\E_8$, $r_{G_0}(\xi)$ is zero.
\end{rem}

Table~\ref{e8.dict} summarizes what we have proved about the relationship
between the Rost invariant, $J$-invariant, and Tits indexes for groups of type
$\E_8$ at the prime $3$. Note that we proved the
equivalent description for Tits indexes and the Rost invariant over fields of arbitrary characteristic,
and the equivalent description for the $J$-invariant only for fields of characteristic $0$.

Note also that for every $k$, there is a versal torsor
$\xi \in H^1(K, G_0)$ for some extension $K/k$, and that for such a $\xi$,
$J_3(G)$ is maximal \cite[p.~1036]{PSZ}, i.e., $(1,1)$.
By Theorem \ref{E8symb}, the $3$-torsion part of $r_{G_0}(\xi)$ is not a symbol in $H^3(L, \Z/3\Z(2))$ for every extension $L/K$ of degree not divisible by $3$.

\begin{table}[hbt]
\begin{tabular}{c|ccc}
$J_3(G)$&$(0,0)$&$(1,0)$&$(1,1)$\\ \hline
\rb{Tits index of $G$}&\rb{split}&
\begin{picture}(8,1.9)
    \multiput(1,0.7)(1,0){7}{\circle*{\darkrad}}
    \multiput(6,0.7)(1,0){2}{\circle{\lrad}}
    \put(3,1.45){\circle*{\darkrad}}

    \put(1,.7){\line(1,0){6}}
    \put(3,1.45){\line(0,-1){0.75}}
\end{picture}
&\rb{anisotropic}\\
$r_{G_0}(\xi)$&$0$&nonzero symbol&otherwise
\end{tabular}
\caption{Dictionary relating the mod-$3$ $J$-invariant of $G$, the Tits index
of $G$ over a $3$-closure $K$ of $k$, and the Rost invariant of $r_{G_0}(\xi_K)$,
for $G_0$ split of type $\E_8$.} \label{e8.dict}
\end{table}

\begin{cor} \label{E8.QpC}
Let $C$ be a smooth projective irreducible curve over $\qq_p$ with $p\ne 3$.
If $G$ is a group of type $\E_8$ over $\qq_p(C)$, then the $G$-variety $X_{7,8}$ is isotropic over a field extension of
degree coprime to $3$.
\end{cor}
\begin{proof}
By \cite[Theorem~3.5]{PaSu10} each element in $H^3(\qq_p(C),\zz/3)$ is a 
symbol over a field extension of degree coprime to $3$. Therefore by
Theorem \ref{E8symb} the variety $X_{7,8}$ is isotropic over a field extension
of degree coprime to $3$.
\end{proof}

\medskip

{\small\noindent{\textbf{Acknowledgements.}}
The authors sincerely thank Charles De Clercq, Wilberd van der Kallen, Geordie Williamson, Maksim Zhykhovich,
and the referees for their remarks.
The first author's research was partially supported by NSF grant DMS-1201542, NSA grant H98230-11-1-0178, Emory University, and the Charles T.~Winship Fund.
The second and the third authors gratefully acknowledge the support of the MPIM Bonn and of
the SFB/Transregio 45 Bonn-Essen-Mainz.
The second author was also supported by JSC ``Gazprom Neft'' and RFBR 13-01-91150, 13-01-00429, 13-01-92699.}

\bibliographystyle{chicago}

\end{document}